\documentclass[11pt, a4]{amsart}

\usepackage{amsmath,amssymb,amsthm,mathrsfs,color}
\usepackage[bbgreekl]{mathbbol}
\usepackage{bbm}
\usepackage{tikz-cd}
\usepackage[all]{xy}
\usepackage{stmaryrd}
\usepackage{hyperref} 
\hypersetup{colorlinks,%
    citecolor=black,%
    filecolor=black,%
   linkcolor=black,%
   urlcolor=black}

\usepackage[shortalphabetic]{amsrefs}
 \makeatletter
\def\append@label@year@{%
    \safe@set\@tempcnta\bib@year
    \edef\bib@citeyear{\the\@tempcnta}%
    \ifnum\bib@citeyear>9
      \append@to@stem{%
          \ifx\bib@year\@empty
          \else
            \@xp\year@short \bib@citeyear \@nil
          \fi
      }%
    \fi    
}
\makeatother

\let\oldtocsection=\tocsection
\renewcommand{\tocsection}[2]{\hspace{0em}\oldtocsection{#1}{#2}}

\usepackage[a4paper]{geometry}

\begin{document}
\setlength{\unitlength}{2.5cm}

\newtheorem{thm}{Theorem}[section]
\newtheorem{lm}[thm]{Lemma}
\newtheorem{prop}[thm]{Proposition}
\newtheorem{cor}[thm]{Corollary}
\newtheorem{conj}[thm]{Conjecture}

\theoremstyle{definition}
\newtheorem{dfn}[thm]{Definition}
\newtheorem{eg}[thm]{Example}
\newtheorem{rmk}[thm]{Remark}

\newcommand{\TA}{\text{A}}
\newcommand{\TB}{\text{B}}
\newcommand{\TC}{\text{C}}
\newcommand{\TD}{\text{D}}
\newcommand{\TE}{\text{E}}
\newcommand{\TF}{\text{F}}
\newcommand{\TG}{\text{G}}
\newcommand{\ii}{\mathbf{i}}

\newcommand{\F}{\mathbf{F}}
\newcommand{\N}{\mathbf{N}}
\newcommand{\R}{\mathbf{R}}
\newcommand{\C}{\mathbf{C}}
\newcommand{\Z}{\mathbf{Z}}
\newcommand{\Q}{\mathbf{Q}}
\newcommand{\T}{\mathbf{T}}
\newcommand{\K}{\mathbf{K}}
\newcommand{\LL}{\mathcal{L}}
\newcommand{\rA}{{\rm A}}
\newcommand{\rC}{{\rm C}}

\newcommand{\G}{\text{G}}
\newcommand{\re}{\text{Re}}
\newcommand{\im}{\text{Im}}
\newcommand{\gal}{\text{Gal}}
\newcommand{\ke}{\text{Ker}}
\newcommand{\ma}{\text{Max}}
\newcommand{\Spec}{\text{Spec}}
\newcommand{\Pic}{\text{Pic}}
\newcommand{\ord}{\text{ord}}
\newcommand{\app}{\thickapprox}
\newcommand{\deh}{H_\mca{D}}
\newcommand{\moh}{H_\mca{M}}
\newcommand{\ab}{\text{ab}}
\newcommand{\Mp}{\text{Mp}}
\newcommand{\Sp}{{\rm Sp}}
\newcommand{\GL}{{\rm GL}}
\newcommand{\PGL}{{\rm PGL}}
\newcommand{\SL}{{\rm SL}}
\newcommand{\Spin}{\text{Spin}}
\newcommand{\Ind}{\text{Ind}}
\newcommand{\Res}{\text{Res}}
\newcommand{\vs}{\vec{s}}
\newcommand{\Hom}{\text{Hom}}
\newcommand{\msc}[1]{\mathscr{#1}}
\newcommand{\mfr}[1]{\mathfrak{#1}}
\newcommand{\mca}[1]{\mathcal{#1}}
\newcommand{\mbf}[1]{\mathbb{#1}}
\newcommand{\wchi}{\wt{\chi}}
\newcommand{\GGMA}{\pmb{\Gamma}}
\newcommand{\into}{\hookrightarrow}
\newcommand{\onto}{\twoheadrightarrow}

\newcommand{\dual}[1]{{}^L\wt{#1}}
\newcommand{\cd}[1]{{\wt{#1}}^\vee}
\newcommand{\noa}{\mathbf{n}_o^\alpha}
\newcommand{\koa}{\mathbf{k}_o^\alpha}
\newcommand{\Hs}{\textsf{Hs}}
\newcommand{\Inc}{\textsf{Inc}}
\newcommand{\CExt}{\textsf{CExt}}
\newcommand{\Bis}{\textsf{Bis}}
\newcommand{\Rec}{\text{Rec}}
\newcommand{\s}{\mathbf{s}}
\newcommand{\cc}{\mathbf{c}}
\newcommand{\bfa}{\mathbf{a}}
\newcommand{\id}{{\rm id}}
\newcommand{\g}{\mathbf{g}_{\psi^{-1}}}
\newcommand{\w}{\mathbbm{w}}
\newcommand{\Ftn}{{\sf Ftn}}
\newcommand{\p}{\mathbf{p}}
\newcommand{\bq}{\mathbf{q}}
\newcommand{\WD}{\text{WD}}
\newcommand{\W}{\text{W}}
\newcommand{\Wh}{{\sf Wh}_\psi}
\newcommand{\ggma}{\pmb{\gamma}}
\newcommand{\sct}{\text{sc}}
\newcommand{\OF}{\mca{O}^\digamma}
\newcommand{\vep}{\pmb{\varepsilon}}
\newcommand{\gk}{c_{\sf gk}}

\newcommand{\cu}[1]{\textsc{\underline{#1}}}
\newcommand{\set}[1]{\left\{#1\right\}}
\newcommand{\ul}[1]{\underline{#1}}
\newcommand{\wt}[1]{\widetilde{#1}}
\newcommand{\angb}[2]{\left\langle #1, #2 \right\rangle}
\newcommand{\seq}[3]{\xymatrix{
#1 \ar@{^(->}[r] & #2 \ar@{>>}[r] & #3
}}
\newcommand{\wm}[1]{\wt{\mbf{#1}}}
\newcommand{\elt}[1]{\pmb{\big[} #1\pmb{\big]} }
\newcommand{\ceil}[1]{\left\lceil #1 \right\rceil}
\newcommand{\val}[1]{\left| #1 \right|}

\title[Bump-Hoffstein conjecture]{Generalized Bump-Hoffstein conjecture for coverings of the general linear groups}

\author{Fan Gao}
\address{Department of Mathematics, Purdue University, 150 N. University Street, West Lafayette, IN 47907}
\email{gaofan.math@gmail.com}
\subjclass[2010]{Primary 11F70; Secondary 22E50}
\keywords{Brylinski-Deligne covering groups, theta representations, Whittaker functionals, Bump-Hoffstein conjecture, dual groups}
\maketitle
\begin{abstract} We investigate the extent to which the Bump-Hoffstein conjecture could be generalized for central coverings of the general linear groups. We provide evidence for such generalized Bump-Hoffstein conjecture by proving some special cases.
\end{abstract}
\tableofcontents

\section{Introduction}
Rankin-Selberg integral has been one of the cornerstones for the theory of $L$-functions for reductive linear algebraic groups. In particular, the uniqueness of Whittaker functionals plays a crucial role for the decomposition of a global integral into product of local integrals. As a consequence, one could extract global $L$-functions from such decomposition. In contrast, for finite degree central coverings of linear algebraic groups, such uniqueness property rarely holds (cf. \cite{KP}, \cite{Ga2} and see however \cite{Szp}), and this has hindered a direct application of the Rankin-Selberg integral in the covering setting.

However, Bump and Hoffstein have made some important observations in their papers \cite{BH1} and \cite{BH2} decades ago. Based on their seminal work, a new way (in the sense of Piatetski-Shapiro and Rallis, see \cite{PSR} and also \cite{BFG}) of extracting global partial $L$-functions is available and investigated in the references cited above. They propose some deep conjectures regarding the Rankin-Selberg convolution of Whittaker models for Kazhdan-Patterson covering groups (cf. \cite{KP}) of the general linear group. 

We note that there are several Bump-Hoffstein conjectures related to Rankin-Selberg integrals; for us we exclusively refer to the specific one  as follows (cf. \cite{BH1}, \cite{BF}):
\begin{enumerate}
\item[$\bullet$] Let $(\wt{\GL}_r^{(n)}, \wt{\GL}_n^{(n)})$ be a pair of Kazhdan-Patterson $n$-fold covering groups with $r< n$. Let $\pi$ be generic unramified genuine representation of $\wt{\GL}_r^{(n)}$. Let $\Theta(\wt{\GL}_n^{(n)}, \chi)$ be a theta representation of $\wt{\GL}_n^{(n)}$ which possesses uniqueness of Whittaker model $\mca{W}^{\GL_n}$, where $\chi$ is a genuine exceptional character. Then, the Rankin-Selberg integral of $\overline{\mca{W}^{\GL_n}}$ against any Whittaker model $\mca{W}^{\GL_r}$ of $\pi$, which might not be unique in general, is equal to a certain $L$-function associated to $\pi$ and $\chi$ multiplied with $\overline{\mca{W}^{\GL_n}}(1) \cdot \mca{W}^{\GL_r}(1)$.
\end{enumerate}
For a more precise form, see the generalized Bump-Hoffstein Conjecture \ref{C:GBH} below.
\vskip 5pt

Following the work of \cite{KP} and \cite{BH1}, T. Suzuki proves in a series of papers the conjecture above when the representation $\pi$ involved is also a theta representation, see \cite{Suz1}, \cite{Suz2}. In fact, he generalizes the conjecture in another direction where $\Theta(\wt{\GL}_n^{(n)}, \chi)$ above is replaced by some theta-like representation which also has unique Whittaker models (cf. \cite{Suz3}). The approach of Suzuki relies on investigating the Shahidi local coefficient matrix for Kazhdan-Patterson covering groups firstly studied in \cite{KP}. Though conceptually clear, with this approach one often confronts with some computational difficulties, and a substantial part of Suzuki's argument relies on a delicate analysis of the matrix mentioned above.

In their excellent exposition \cite{BF}, Bump and Friedberg propose a way of attacking the Bump-Hoffstein conjecture from the viewpoint of generating functions. This dramatically reduces the explicit computation required as in the work of Suzuki. By \cite{BF}, the proof of Bump-Hoffstein conjecture above is reduced to a certain statement for integral representation of Whittaker values, which is completely proved recently by D. Ginzburg (\cite{Gin}). Therefore, the Bump-Hoffstein conjecture for Kazhdan-Patterson covering groups is proven by \cite{BF} coupled with \cite{Gin}.

We note that all the work mentioned above focuses only on Kazhdan-Patterson covering groups of $\GL_r$. In this paper, we investigate the extent to which the Bump-Hoffstein conjecture is expected to hold in general. Such consideration is motivated from the following observations and facts:

\begin{enumerate}
\item[$\bullet$] The structural theory of quite general covering groups of $\GL_r$ is developed and studied in \cite{BD}. Moreover, the theories of dual group  (cf. \cite{FL}, \cite{Mc1}, \cite{Re}) and $L$-group (cf. \cite{We2}) are also developed.  Besides this, some fundamental analysis is carried out in \cite{We2} and \cite{GG}.  Large part of the structural theory depends on combinatorial data, and thus amenable to explicit analysis.
\item[$\bullet$] For general coverings of a reductive linear algebraic group, the dimension of Whittaker functionals for theta representations is studied in details in \cite{Ga2}. In particular, for coverings of $\GL_r$, we could completely determine the dimension for their theta representations. Distinguished theta representations, which are analogues of $\Theta(\wt{\GL}_n^{(n)}, \chi)$ above, will be the key input, if one would like to formulate a Bump-Hoffstein conjecture for general covering groups of $\GL_r$.
\item[$\bullet$] A careful analysis of the argument in Suzuki gives hints that, at least for theta representations, the main argument should be adaptable to general coverings of $\GL_r$.
\end{enumerate}

Therefore, in this paper we consider the generalized Bump-Hoffstein and provide some evidence for its validity. We now briefly explain the content of this paper.

In \S 2, we outline the structural facts for general $n$-fold coverings of $\GL_r$, following quite closely \cite{Ga2}. We introduce the notion of a fundamental pair $(\wt{\GL}_r, \wt{\GL}_R ), r<R$ of $n$-fold covering groups, see Definition \ref{D:FP}. One special property is that the dual group of $\wt{\GL}_R$ from a fundamental pair is always $\mbf{GL}_R$.

In \S 3, we introduce Whittaker models for coverings of $\wt{\GL}_r$, recall and derive some results from \cite{Ga2}. We formulate the generalized Bump-Hoffstein Conjecture \ref{C:GBH} for a fundamental pair. Along the way, we prove some results on the Whittaker functions for theta representations.

In \S 4, we consider a fundamental pair $(\wt{\GL}_r^{(n)}, \wt{\GL}_R^{(n)} )$ and theta representations $\Theta(\wt{\GL}_r^{(n)}, \mu)$ and $\Theta(\wt{\GL}_R^{(n)}, \chi)$ for the two groups respectively. We show that the generalized Bump-Hoffstein conjecture holds in this case (i.e. for theta representations). Here $\Theta(\wt{\GL}_R^{(n)}, \chi)$ has uniqueness of Whittaker models, whereas $\Theta(\wt{\GL}_r^{(n)}, \mu)$ does not in general, albeit generic. Part of our proof follows closely that of Suzuki in \cite{Suz1}, and it relies crucially on the fact that the values of Whittaker functions for theta representations are represented by certain Gauss sums, proved in \S 3. The main result is Theorem \ref{T:theta}.

In \S 5, we consider the rank two case with $r=2$. That is, we consider a fundamental pair $(\wt{\GL}_2^{(n)}, \wt{\GL}_R^{(n)} )$ with $2< R$, and correspondingly a pair of representations $(\pi, \Theta(\wt{\GL}_R^{(n)}, \chi))$, where $\pi$ is any generic unramified representation of $\wt{\GL}_2^{(n)}$. We show that the Bump-Hoffstein conjecture holds in this case as well. The main result is summarized in Theorem \ref{T:rank2}.

 In the last section, we will mention in passing the work of Bump-Friedberg \cite{BF} and highlight on the analogous formulation for general coverings of $\GL_r$ which could be used to prove the generalized Bump-Hoffstein conjecture. The discussion in \cite{BF} could be carried in parallel for a fundamental pair in this paper, though we will refrain from giving duplicated details. Moreover, it is believed that the strategy in \cite{Gin} of proving the Bump-Hoffstein conjecture for Kazhdan-Patterson covering groups should be adaptable in the general setting. However, for lack of expertise, we will give no elaborate discussion on such connections.

We note that it is natural to inquire about the most general framework, which might lie beyond the notion of fundamental pair, and in which analogous Bump-Hoffstein conjecture holds. For instance, one might consider a certain pair $(\wt{\GL}_r, \wt{\GL}_R)$ of $n$-fold coverings groups such that the dual group of $\wt{\GL}_R$ is $\GL_R$ and its theta representation is distinguished (i.e. with unique Whittaker model). However, in the last section, we remark by considering certain Kazhdan-Patterson coverings that such a pair, albeit natural, does not seem to fit in a Bump-Hoffstein framework. In this regard, the constraints imposed on a fundamental pair seem to be sharp.
\vskip 5pt

We do not investigate any related global problem in this paper. However, the consideration of Conjecture \ref{C:GBH} is mainly motivated from studying the automorphic $L$-functions for $\wt{\GL}_r$ in the global context, see \cite{BF} and \cite{Gin}. The raison d'\^etre of the article is to present explicitly the results of interest, part of which might be known to experts already. We provide details for the computations, as some of the formulation (to the best of our knowledge) has not been made explicitly for general coverings of $\GL_r$. The reader could use the Kazhdan-Patterson covering groups (cf. \S \ref{S:KP}) and especially the Savin covering groups (cf. \S \ref{S:Savin}) as running examples for some new phenomena investigated in this paper.
\vskip 5pt

\section{Covering groups of $\GL_r$}
Let $F$ be a finite extension of $\Q_p$ with residue field of size $q$. Fix a uniformizer $\varpi$ of $F$. Contrary to the most natural assumption that $F^\times$ contains the full group $\bbmu_{n}$ of $n$-th roots of unity, we will assume through the paper $\bbmu_{2n} \subseteq F^\times$ to simplify some of the computation. We also assume that $p$ does not divide $2n$, and consequently $q\equiv 1 \mod 2n$ in this case.

For a brief introduction to covering groups with summarized results, we refer to \cite{Ga2}. We will recall below only some notations. Consider the general linear group $\mbf{GL}_r$ with root data 
$$\set{ X_{\GL_r},\ \Delta_{\GL_r},  \ \Phi_{\GL_r}; \ Y_{\GL_r}, \ \Delta_{\GL_r}^\vee, \ \Phi_{\GL_r}^\vee  }.$$
Let $\mbf{T}$ be a maximal split torus of $\mbf{GL}_r$. Here $\set{e_1, e_2, ..., e_r}$ is a basis for the cocharacter lattice $Y_{\GL_r}$ of $\mbf{T}$, and $\set{e_1^*, e_2^*, ..., e_r^*}$ a basis for the character lattice of $\mbf{T}$ such that the pairing $\angb{e_i^*}{e_j}=\delta_{ij}$.  Here $\Phi_{\GL_r}^\vee$ and $\Phi_{\GL_r}$ denote the coroots and roots respectively. Denote $\alpha_i^\vee:=e_i- e_{i+1}$ and $\alpha_i:= e_i^* - e_{i+1}^*$. We choose simple coroots $\Delta_{\GL_r}^\vee=\set{\alpha_i^\vee: 1\le i\le r-1}$ and simple roots $\Delta_{\GL_r}=\set{\alpha_i: 1\le i\le r-1}$.  Let $\mbf{B}=\mbf{T} \mbf{U}$ be the Borel subgroup associated with $\Delta_{\GL_r}$.
Let
$$\set{ e_{\alpha}: \mbf{G}_\text{add} \to \mbf{U}_\alpha}_{\alpha \in \Phi_{\GL_r}}$$
be a Chevalley-Steinberg system of pinnings for $\mbf{GL}_r$. Denote by $W_{\GL_r}$ the Weyl group of $(\mbf{GL}_r, \mbf{T})$ which we identify with the Weyl group of the coroot system. In particular, $W_{\GL_r}$ is generated by simple reflections $\set{\w_\alpha: \alpha^\vee \in \Delta_{\GL_r}^\vee}$ for $Y_{\GL_r}\otimes \Q$.

The isomorphism classes of $\mbf{K}_2$-extensions of $\mbf{GL}_r$ incarnated by the pairs $(D,\eta)$ (cf. \cite{GG}, \cite{Ga2}) are determined by Weyl-invariant integer-valued quadratic forms on $Y_{\GL_r}$. Let $Q$ be such a quadratic form and denote by $B_Q(y_1,y_2):=Q(y_1+y_2)-Q(y_1)-Q(y_2)$ the associated bilinear form. For $\mbf{GL}_r$, any Weyl-invariant integer-valued bilinear form $B_Q$ is determined by
\begin{equation} \label{BQ}
B_Q(e_i, e_i)=2\p \text{ and } B_Q(e_i, e_j)=\bq \text{ if } i\ne j,
\end{equation}
where $\p, \bq \in \Z$ are any two integers. For any coroot $\alpha^\vee \in \Phi^\vee_{\GL_r}$, one has $$Q(\alpha^\vee)=2\p-\bq.$$ 

\subsection{Families of $\mbf{K}_2$-extensions of $\mbf{GL}_r$}
We write $\wm{GL}_r[\p, \bq]$ for the isomorphism class of $\mbf{K}_2$-extension of $\mbf{GL}_r$ arising from the bilinear form with parameters $\p, \bq$ as in (\ref{BQ}). For any integral equation
$$\text{(EQ)}: \quad A\p+B\bq=C,$$
where $A, B, C\in \Z$, denote by $\wm{GL}_r\text{(EQ)}$ the family of all coverings $\wm{GL}_r[\p, \bq]$ with $(\p, \bq)$ satisfying (EQ). It is natural to take $A=2, B=-1$ since $Q(\alpha^\vee)=2\p-\bq$. There are three families of extensions we would like to highlight.

\subsubsection{The twistor coverings} \label{S:TC}
Consider the integral equation
$$\text{(TC)}: \quad 2\p+(-1)\bq=0$$
and the arising family $\wm{GL}_r(\text{TC})$. An alternative description of $\wm{GL}_r[\p,2\p]$ in the family is as the pull-back from $\wm{GL}_1[\p]$ parametrized by $Q(e_1)=\p$ via the determinant map:
$$\xymatrix{
\mbf{K}_2 \ar@{^(->}[r] & \wm{GL}_1[\p] \ar@{>>}[r] & \mbf{GL}_1 \\
\mbf{K}_2 \ar@{^(->}[r] \ar[u]^-= & \wm{GL}_r[\p, 2\p] \ar@{>>}[r] \ar[u] & \mbf{GL}_r \ar[u]^-{\text{det}}.
}$$
Any extension $\wm{GL}_r[\p,2\p]$ restricted to $\mbf{SL}_r \subseteq \mbf{GL}_r$ gives rise to the trivial extension $\mbf{K}_2\times \mbf{SL}_r$ over $\mbf{SL}_r$.

\subsubsection{The Kazhdan-Patterson coverings} \label{S:KP}
Let (KP) be the integral equation
$$\text{(KP)}: \quad 2\p+(-1)\bq=-1,$$
from which we obtain a family $\wm{GL}_r(\text{KP})$. The $n$-fold covering groups arising from this family are exactly those studied by Kazhdan-Patterson (cf. \cite{KP}). The parameter $\p$ is just the twisting parameter $c$ in the notation of \cite{KP}. This family is the most widely studied one among all Brylinski-Deligne extensions of $\mbf{GL}_r$.

The family $\wm{GL}_r(\text{KP})$ restricts to give the same extension over $\mbf{SL}_r$ with fundamental invariant $Q(\alpha^\vee)=-1$ for any coroot $\alpha^\vee$ of $\mbf{SL}_r$. The group $\wm{GL}_r[0, 1]$ is the untwisted covering in \cite{KP} (since $\p=0$), and is also the focus of the earlier works of \cite{GHPS}, \cite{Suz1}, \cite{Suz2} and \cite{Suz3} etc.

On the other hand, let $f: \mbf{GL}_r \to \mbf{SL}_{r+1}$ be the embedding given by $g\mapsto (g, \det(g)^{-1})$. Let $\wm{SL}_{r+1}$ be the extension with  invariant $Q(\alpha^\vee)=-1$ for any coroot $\alpha^\vee$. Then
$$\wm{GL}_r[-1, -1] \simeq f^*(\wm{SL}_{r+1}),$$
the pull-back of $\wm{SL}_{r+1}$ to $\mbf{GL}_r$. Thus $\wm{GL}_r[-1, -1]$ arises from $\wm{GL}_r[0, 1]$ ``twisted" by the twistor covering group $\wm{GL}_r[1, 2]$ in the Brylinski-Deligne category. Here by twisting, we mean the addition of bilinear forms.

\subsubsection{The Savin coverings} \label{S:Savin}
Let (S) be the integral equation
$$\text{(S)}: \quad 2\p+(-1)\bq=-2,$$
which gives rise to the family $\wm{GL}_r(\text{S})$. Again, there are two natural extensions in the family. The first one is $\wm{GL}_r[0, 2]$, which conveniently might be viewed as the untwsited covering in the family $\wm{GL}_r(\text{S})$. On the other hand,  let $h: \mbf{GL}_r \to \mbf{Sp}_{2r}$ be the usual embedding of $\mbf{GL}_r$ as the Siegel Levi subgroup of $\mbf{Sp}_{2r}$. Let $\wm{Sp}_{2r}$ be the $\mbf{K}_2$-extension with fundamental invariant $Q(\alpha^\vee)=-1$ for any short coroot $\alpha^\vee$ of $\mbf{Sp}_{2r}$. Then
$$\wm{GL}_r[-1, 0] \simeq h^*(\wm{Sp}_{2r}),$$
which is also obtained from $\wm{GL}_r[0, 2]$ by a twist by the covering $\wm{GL}_r[1, 2]$ in the Brylinski-Deligne category. 

The fact $\bq=0$ for $\wm{GL}_r[-1,0]$ accounts for the block-commutativity for the covering Levi subgroups of $\wm{GL}_r[-1, 0]$. The group $\wm{GL}_r[-1, 0]$, to the best of our knowledge, is firstly studied by G. Savin (cf. \cite{Sa}). Therefore, in this paper by \emph{special} Savin covering we will refer to $\wm{GL}_r[-1,0]$. 

Note also that any member in $\wm{GL}_r(\text{S})$ restricts to give the same extension over $\mbf{SL}_r$ with fundamental invariant $Q(\alpha^\vee)=-2$.

\subsection{Topological covers}
For simplicity of notation, we will write $\wm{GL}_r$ for $\wm{GL}_r[\p,\bq]$ with the structural parameters $\p$ and $\bq$ understood. A topological covering group, in the sense of Weissman \cite{We2}, is simply a pair $(\wm{GL}_r, n)$. More explicitly, $(\wm{GL}_r, n)$ gives rise to an $n$-fold topological covering group $\wt{\GL}_r$ of $\GL_r:=\mbf{GL}_r(F)$, from the push-out of the $n$-th Hilbert symbol $(-,-)_n: \mbf{K}_2(F) \to \bbmu_n$. We will also write $\wt{\GL}_r^{(n)}$ whenever we would like to emphasize the degree of covering. Let $\wt{T}$ be the covering torus of $T:=\mbf{T}(F)$, and $\wt{B}$ the covering of the Borel subgroup $B\subset \GL_r$. We identify $\bbmu_n$ as a subgroup of $\C^\times$ by an embedding $\iota: \bbmu_n \into \C^\times$.

Let $K \subseteq \GL_r$ be the maximal compact subgroup arising from the pinnings $\set{e_\alpha: \alpha\in \Phi_{\GL_r}}$. Since by assumption $\text{gcd}(p, n)=1$, the covering group $\wt{\GL}_r$ splits over $K$ (cf. \cite{GG}), and we fix such a splitting. To describe $\wt{\GL}_r$, we choose (without loss of generality on the isomorphism class) a bisector $D$ of the symmetric bilinear form $B_Q$ as follows:
\begin{equation} \label{F:D}
D(e_i, e_j) =
\begin{cases}
0 & \text{ if } i< j, \\
Q(e_i) & \text{ if } i=j, \\
B_Q(e_i, e_j) & \text{ if } i> j.
\end{cases}
\end{equation}
We can describe the group structure of $\wt{\GL}_r$ as follows.

First, the group $\wt{\GL}_r$ splits canonically over any unipotent element of $\GL_r$. Denote by $\wt{e}_\alpha(a) \in \wt{\GL}_r$ for $\alpha\in \Phi_{\GL_r}$ and $a\in F$ the canonical lifting of $e_\alpha(a) \in \GL_r$. For $\alpha\in \Phi_{\GL_r}$ and $a\in F^\times$, define 
\begin{equation} \label{F:w}
w_\alpha(a):=e_{\alpha}(a) \cdot e_{-\alpha}(-a^{-1}) \cdot e_{\alpha}(a) \text{ and } \wt{w}_\alpha(a):=\wt{e}_{\alpha}(a) \cdot \wt{e}_{-\alpha}(-a^{-1}) \cdot \wt{e}_{\alpha}(a).
\end{equation}
This gives a natural representative $w_\alpha:=w_\alpha(1)$ in $K$, and also $\wt{w}_\alpha:=\wt{w}_\alpha (1)$ in $\wt{K}$, of the Weyl element $\w_\alpha\in W_{\GL_r}$. Moreover, for any $h_\alpha(a):=\alpha^\vee(a)\in T$, there is a natural lifting 
\begin{equation} \label{h-alpha}
\wt{h}_\alpha(a):=\wt{w}_\alpha(a)\cdot \wt{w}_\alpha(-1) \in \wt{T},
\end{equation}
which depends only on the pinnings and the canonical unipotent splitting.

Second, there is a section $\s$ of $T$ into $\wt{T}$ such that the group law on $\wt{T}$ is given by
\begin{equation} \label{F:s}
\s(y_1(a)) \cdot \s(y_2(b)) = (a, b)_n^{D(y_1, y_2)} \cdot \s(y_1(a)\cdot y_2(b)).
\end{equation}
Moreover, for the natural lifting $\wt{h}_\alpha(a)$ of $h_\alpha(a)$ above, one has
\begin{equation*} 
\wt{h}_\alpha(a)= \s(h_\alpha(a)) \in \wt{T}.
\end{equation*}
We could assume that the restriction of $\s$ on $T\cap K$ agrees with the splitting of $K$ into $\wt{\GL}_r$. For convenience, for any $y\in Y_{\GL_r}$ we will use the notation
$$\s_y:=\s(y(\varpi)) \in \wt{T}.$$

Third, let $w_\alpha \in K$ be the above natural representative of $\w_\alpha\in W_{\GL_r}$. For any $\wt{y(a)} \in \wt{T}$, one has
\begin{equation} \label{F:W-act}
w_\alpha \cdot \wt{y(a)} \cdot w_\alpha^{-1} = \wt{y(a)} \cdot \wt{h}_\alpha(a^{-\angb{y}{\alpha}}),
\end{equation}
where $\angb{-}{-}$ is the paring between $Y_{\GL_r}$ and $X_{\GL_r}$.

\vskip 5pt
The following result is one of the main reasons for our assumption $\bbmu_{2n} \subseteq F^\times$.
\begin{lm}  \label{L:2n}
Let $\w=\w_k ... \w_2 \w_1 \in W_{\GL_r}$ be a minimal decomposition of $\w\in W_{\GL_r}$ with each $\w_i=\w_{\alpha_i}$ for some $\alpha_i \in \Delta_{\GL_r}$. Let $w:=w_k ... w_2 w_1 \in K$ be the representative of $\w$, where $w_i$ is as in (\ref{F:w}). Then we have
$$w\cdot \s_y \cdot w^{-1} = \s_{\w(y)}$$
for all $y\in Y_{\GL_r}$.
\end{lm}
\begin{proof}
By induction, it suffices to show the case $\w=\w_\alpha$ for some $\alpha\in \Delta_{\GL_r}$. For this, we have
\begin{align*}
& w_\alpha \cdot \s_y \cdot w_\alpha^{-1} \\
=& \s(\varpi^y) \cdot \wt{h_\alpha}(\varpi^{-\angb{y}{\alpha}}) \text{ by (\ref{F:W-act}) } \\
=& \s(\varpi^y) \cdot \s(\varpi^{-\angb{y}{\alpha}\alpha^\vee}) \\
=& (\varpi, \varpi)_n^{\angb{y}{\alpha} \cdot D(y, \alpha^\vee)} \cdot \s_{\w_\alpha(y)} \text{ by (\ref{F:s})} \\
=& \s_{\w_\alpha(y)},
\end{align*}
where the last equality follows from $(\varpi, \varpi)_n=1$, since we have assumed $\bbmu_{2n} \subseteq F^\times$.
\end{proof}

\subsection{Dual group}
From the bilinear form $B_Q$ in (\ref{BQ}), we consider the lattice $Y_{\GL_r, Q,n}$ given by
$$Y_{\GL_r, Q, n}:=\set{y\in Y_{\GL_r}: B_Q(y, e_i) \in n\Z \text{ for all } i} \subseteq Y_{\GL_r}.$$
An explicit computation gives
\begin{equation} \label{YQn}
Y_{\GL_r, Q,n }=\set{\sum_{i=1}^r k_i e_i \in Y_{\GL_r}:  \ Q(\alpha^\vee) \cdot k_j + \bq\cdot \Big(\sum_{i=1}^r k_i \Big) \in  n\Z \text{ for all } j}.
\end{equation}
For $\alpha\in \Phi_{\GL_r}$, write 
$$n_\alpha:=\frac{n}{\text{gcd}(n, Q(\alpha^\vee))}, \quad \alpha_{Q,n}^\vee=n_\alpha \cdot \alpha^\vee, \quad \alpha_{Q,n}=n_\alpha^{-1} \alpha.$$
Let $Y_{\GL_r, Q,n}^{sc} \subseteq Y_{\GL_r, Q,n}$ be the sublattice generated by $\set{\alpha^\vee_{Q,n}: \alpha\in \Phi_{\GL_r}}$.

\vskip 5pt
The complex dual group $\wt{\GL}_r^\vee$ for $\wt{\GL}_r$ has root data (cf. \cite{FL},  \cite{Mc1}, \cite{Re} and \cite{We2}) 
$$\left( Y_{\GL_r, Q,n}, \ \set{\alpha_{Q,n}^\vee}; \ \text{Hom}_\Z(Y_{\GL_r, Q,n}, \Z),\  \set{\alpha_{Q,n}} \right).$$
In particular, $Y_{\GL_r,Q,n}$ is the character lattice and $Y_{\GL_r, Q,n}^{sc}$ the root lattice for $\wt{\GL}_r^\vee$.

\begin{eg} \label{E:KP}
(Kazhdan-Patterson coverings) Consider the Kazhdan-Patterson $n$-fold covering group $\wt{\GL}_r$ with $Q(\alpha^\vee)=2\p-\bq=-1$. In this case, 
one has 
$$Y_{\GL_r, Q, n}=\set{\sum_{i=1}^r k_i e_i: \ k_1 \equiv k_2 \equiv ... \equiv k_r \mod n, \text{ and } n| (\bq r-1)k_i \text{ for all } i}.$$ 
The dual for $\wt{\GL}_r$ is given by (cf. \cite{GG})
$$\wt{\GL}_r^\vee \simeq \set{(g, \lambda) \in  \mbf{GL}_r \times \mbf{GL}_1: \det(g)= \lambda^{\gcd(\bq r-1, n)}} \subset \mbf{GL}_r \times \mbf{GL}_1.$$
In general, the dual group is not $\mbf{GL}_r$. However, if $\gcd(\bq r-1, n)=1$ (for example when $n=r$), then the dual group is $\mbf{GL}_r$. For more details, we refer the reader to \cite{GG}.
\end{eg}

\begin{eg} \label{E:S}
(Savin coverings)
Now we consider the special Savin $n$-fold covering group $\wt{\GL}_r$ arising from $\wm{GL}_r[-1, 0]$, i.e. $\p=-1$ and $\bq=0$. In this case, $Y_{\GL_r, Q, n}=n_\alpha \cdot Y_{\GL_r}$ and the dual group for such $\wt{\GL}_r$ is always $\mbf{GL}_r$. From this, it is conceivable that the special Savin coverings should behave better and are more accessible to investigation.
\end{eg}

\subsection{Fundamental pair}
For $r < R$, denote by 
$$\phi: \mbf{GL}_r \to \mbf{GL}_R$$
the natural embedding such that (by abuse of notation) one has the embedding
$$\phi: Y_{\GL_r} \to Y_{\GL_R}$$
with $\phi(e_i)=e_i$ for $1\le i\le r$.

\begin{dfn} \label{D:FP}
A pair of degree $n$ covering groups $(\wt{\GL}_r^{(n)}, \wt{\GL}_R^{(n)})$ with $r< R$ is called a fundamental pair if the following conditions are satisfied:
\begin{enumerate}
\item[(\textsf{FP1})] Both $\wt{\GL}_r^{(n)}$ and $\wt{\GL}_R^{(n)}$ arise from a symmetric bilinear form with the same parameters $\p, \bq$ as in (\ref{BQ}). That is, we may assume that the $\mbf{K}_2$-extension $\wm{GL}_r$ is pull-back from the $\mbf{K}_2$-extension $\wm{GL}_R[\p,\bq]$ via the embedding $\phi: \mbf{GL}_r \to \mbf{GL}_R$.
\item[(\textsf{FP2})] The two equalities $n_\alpha=R$ and $Y_{\GL_R, Q,n}=n_\alpha \cdot Y_{\GL_R}$ hold.
\end{enumerate}
\end{dfn}
\vskip 5pt

\begin{rmk}
The conditions we impose here are not the most general ones. For instance, instead of (\textsf{FP1}) we could simply require that the \emph{topological} cover $\wt{\GL}_r^{(n)}$ is isomorphic to the pull-back $\phi^*(\wt{\GL}_R^{(n)})$ from $\wt{\GL}_R^{(n)}$. In any case, the conditon ({\sf FP1}) is a formal requirement on compatibility of the two groups. The second condition ({\sf FP2}) plays a key role in the paper.
\end{rmk}
\vskip 10pt

It follows from property (\textsf{FP2}) that the dual group of $\wt{\GL}_R^{(n)}$ from a fundamental pair has root data
$$\left( n_\alpha Y_{\GL_R}, \quad \set{\alpha_{Q,n}^\vee}, \quad n_\alpha^{-1} X_{\GL_R}, \quad \set{\alpha_{Q,n}} \right).$$
Thus, the dual group for the $\wt{\GL}_R^{(n)}$ is always $\mbf{GL}_R$. However, in general the dual group of $\wt{\GL}_r^{(n)}$ in a fundamental pair is not $\mbf{GL}_r$.

\begin{eg}
First, there does not exist any fundamental pair $(\wt{\GL}_r^{(n)}, \wt{\GL}_R^{(n)})$ from twistor coverings, for which one always has $n_\alpha=1$.

Second, any pair of Kazhdan-Patterson covering groups $(\wt{\GL}_r^{(n)}[\p, \bq], \wt{\GL}_n^{(n)}[\p, \bq])$ with $r< n$, where $2\p-\bq=-1$, is a fundamental pair. This is the context for the works of \cite{BH1}, \cite{BF}, \cite{Suz1}-\cite{Suz3}.

Third, consider the special Savin covering group $\wt{\GL}_R^{(n)}$ with $\bq=0$ and $R=n/\text{gcd}(2, n)$. Then, $(\wt{\GL}_r^{(n)}, \wt{\GL}_R^{(n)})$ is a fundamental pair for any $r< R$, where $\wt{\GL}_r^{(n)}$ is the special Savin covering group with $\bq=0$.
\end{eg}

In fact, we have a complete understanding of possible fundamental pairs as follows.
\begin{prop} \label{P:key-d}
Let $\wt{\GL}_r^{(n)}[\p,\bq], r\ge 1$ be an arbitrary covering group with structure parameters $\p$ and $\bq$. Then, it fits into a fundamental pair $(\wt{\GL}_r^{(n)}[\p,\bq],\ \wt{\GL}_R^{(n)}[\p,\bq])$ if and only if 
\begin{enumerate}
\item[$\bullet$] $n_\alpha > r$, and
\item[$\bullet$] $n| (\bq \cdot n_\alpha)$.
\end{enumerate}
If the above conditions are satisfied, then $\wt{\GL}_R^{(n)}[\p,\bq])$ is uniquely determined. For a fundamental pair $(\wt{\GL}_r^{(n)}, \wt{\GL}_R^{(n)})$, one always has
\begin{equation} \label{F:nY}
Y_{\GL_r, Q,n}^{sc} \subseteq n_\alpha \cdot Y_{\GL_r} \subseteq Y_{\GL_r, Q, n}.
\end{equation}
\end{prop}
\begin{proof}
Fix the covering group $\wt{\GL}_r^{(n)}$ associated with $\p$ and $\bq$. If it fits in a fundamental pair, then from definition, $\wt{\GL}_R^{(n)}$ is uniquely determined and is just $\wt{\GL}_{n_\alpha}^{(n)}[\p,\bq]$. Also, such $\wt{\GL}_R^{(n)}$ exists if and only if $n_\alpha > r$ and $Y_{\GL_{n_\alpha}, Q, n}=n_\alpha \cdot Y_{\GL_{n_\alpha}}$. By (\ref{YQn}), we have
$$Y_{\GL_{n_\alpha}, Q,n }=\set{\sum_{i=1}^{n_\alpha} k_i e_i \in Y_{\GL_{n_\alpha}}:  \ Q(\alpha^\vee) \cdot k_j + \bq\cdot \Big(\sum_{i=1}^{n_\alpha} k_i \Big) \in  n\Z \text{ for all } j}.$$
Now if $Y_{\GL_{n_\alpha}, Q, n}=n_\alpha \cdot Y_{\GL_{n_\alpha}}$, then in particular $n_\alpha e_1 \in Y_{\GL_{n_\alpha}, Q, n}$. Therefore $n|(\bq \cdot n_\alpha)$.

Conversely, assume $n|(\bq\cdot n_\alpha)$. Any element $\sum_{i}^{n_\alpha} k_i e_i$ lies in $Y_{\GL_{n_\alpha}, Q, n}$ if and only if
\begin{enumerate}
\item[$\bullet$] for all $i$, one has $k_i=n_\alpha x_i + m$ for some $x_i$, and
\item[$\bullet$] $Q(\alpha^\vee)(n_\alpha x_j + m) + \bq \cdot \left(n_\alpha (\sum_{i=1}^{n_\alpha} x_i) + n_\alpha \cdot m \right)$ lies in $n\Z$ for all $j$.
\end{enumerate} 
Since $n|(\bq\cdot n_\alpha)$, the second condition is equivalent to $n_\alpha |m$. That is, $Y_{\GL_{n_\alpha}, Q, n} = n_\alpha \cdot Y_{\GL_{n_\alpha}}$ in this case. 

Lastly, for (\ref{F:nY}), it suffices to show $n_\alpha \cdot Y_{\GL_r, Q,n} \subseteq Y_{\GL_r, Q,n}$. However, this follows from a straightforward checking by using (\ref{YQn}) and the condition $n|(\bq\cdot n_\alpha)$. This completes the proof.
\end{proof}

\vskip 5pt
\subsection{Weyl action} \label{S:W-act}
We denote by $\w(y)$ the usual Weyl action of $\w\in W_{\GL_r}$ on $y\in Y_{\GL_r}$. Let
$$\rho_{\GL_r}=\frac{1}{2} \sum_{\substack{\alpha^\vee \in \Phi_{\GL_r}^\vee\\ \alpha^\vee >0}} \alpha^\vee$$
be the half sum of the positive coroots. For any $y\in Y_{\GL_r}$, we will write in this paper
$$y_\rho:=y-\rho_{\GL_r} \in Y_{\GL_r}\otimes \Q.$$
By transport of structure, one has the twisted action
$$\w[y]:=\w(y-\rho_{\GL_r}) + \rho_{\GL_r} \in Y_{\GL_r}.$$
For any $y\in Y_{\GL_r}$, we write $\mca{O}_y$ for the Weyl-orbit of $y$  with respect to this twisted action. Let $\OF$ be the set of all free Weyl-orbits in $Y_{\GL_r}$.

Let  $\wp^{sc}: Y_{\GL_r} \to Y_{\GL_r}/Y_{\GL_r,Q,n}^{sc}$ and
\begin{equation} \label{F:p-map}
\wp: Y_{\GL_r} \to Y_{\GL_r}/Y_{\GL_r,Q,n}
\end{equation}
be the two quotient maps. We call $\mca{O}_y$ a $Y_{\GL_r,Q,n}$-free (resp. $Y_{\GL_r, Q,n}^{sc}$-free ) orbit if 
$|\mca{O}_y|=|\wp(\mca{O}_y)|$ (resp. $|\mca{O}_y|=|\wp^{sc}(\mca{O}_y)|$). Write
$$\OF_{Q,n}:=\set{\mca{O}_y \in \OF: \mca{O}_y \text{ is $Y_{\GL_r, Q,n}$-free } }$$
and 
$$\OF_{Q,n, sc}:=\set{\mca{O}_y \in \OF: \mca{O}_y \text{ is $Y_{\GL_r, Q,n}^{sc}$-free } }.$$
 
\vskip 5pt
For $r< R$, by virtue of the embedding $\phi: \mbf{GL}_r \to \mbf{GL}_R$, we have the natural embedding
$$\phi: W_{\GL_r} \to W_{\GL_R}.$$
For any $\w \in W_{\GL_r}$ and $y\in Y_{\GL_r}$, we write $\phi(\w)$ and $\phi(y)$ respectively to emphasize that we view them as associated to $\mbf{GL}_R$. Since $\rho_{\GL_R} -\rho_{\GL_r}$ is fixed by any $\w \in W_{\GL_r}$, it follows that
$$\w(y-\rho_{\GL_r})+\rho_{\GL_r}=\phi(\w)(y-\rho_{\GL_R}) + \rho_{\GL_R};$$
or equivalently, 
$$\w[y]=\phi(\w)[y].$$
Therefore, $\phi: Y_{\GL_r} \to Y_{\GL_R}$ is equivariant with respect to the twisted $W_{\GL_r}$-actions.

\vskip 5pt
Let $\w_{0,r}$ and $\w_{0,R}$ be the longest Weyl element of $W_{\GL_r}$ and $W_{\GL_R}$ respectively. For any $\w \in W_{\GL_r}$, we write
$$\widehat{\w}:=\w_{0,r} \cdot \w \cdot \w_{0,r}^{-1}  \in W_{\GL_r}.$$
We will also use the notation
$$\widehat{\phi(\w)}:= \w_{0,R} \cdot \phi(\w) \cdot \w_{0,R}^{-1} \in W_{\GL_R}.$$
Similarly, for any $y\in Y_{\GL_r}$, denote
$$\widehat{y}:=\w_{0,r}(y) \in Y_{\GL_r}, \quad \widehat{\phi(y)}:=\w_{0,R}(y) \in Y_{\GL_R}.$$
There will be no confusion for the notations. 

For any $\alpha^\vee \in \Delta_{\GL_r}^\vee$ and associated $\w_\alpha \in W_{\GL_r}$, the first lemma below is immediate.

\begin{lm} \label{L:01}
For any $\alpha^\vee \in \Delta_{\GL_r}^\vee$, we have $\widehat{\w_\alpha}=\w_\beta$ where $\beta^\vee=-\w_{0,r}(\alpha^\vee)\in \Delta_{\GL_r}^\vee$.
\end{lm}

\begin{lm} \label{L:02}
For any $y\in Y_{\GL_r}$, the following statements are equivalent:
\begin{enumerate}
\item[(i)] $\widehat{\phi(y)}=\widehat{\phi(\w)}[0] + y'$ for some $\w\in W_{\GL_r}$ and $y'\in n_\alpha \cdot Y_{\GL_R}$;
\item[(ii)] $y=-\w[0] + y_1$ for some $\w \in W_{\GL_r}$ and $y_1 \in n_\alpha \cdot Y_{\GL_r}$;
\item[(iii)] $\widehat{y}=\widehat{\w}[0] + y_2$ for some $\w \in W_{\GL_r}$ and $y_2\in n_\alpha \cdot Y_{\GL_r}$.
\end{enumerate}
If one and thus all of the above hold, then
$$y'=\widehat{\phi(y_1)}, \quad y_2=\widehat{y}.$$
\end{lm}
\begin{proof} We show the equivalence between (i) and (ii), while that with (iii) follows similarly. For (i), the equality $\widehat{\phi(y)}=\widehat{\phi(\w)}[0] + y'$ is equivalent to 
$$\w_{0,R}(y)=\w_{0,R}\cdot \w \cdot \w_{0,R}^{-1}(0-\rho_{\GL_R}) + \rho_{\GL_R} + y',$$
which is further simplified to
$$y=\w(\rho_{\GL_R}) -\rho_{\GL_R} + \w_{0,R}(y')=\w(\rho_{\GL_r}) -\rho_{\GL_r}  + \w_{0,R}(y')=-\w[0] + \w_{0,R}(y').$$
However, since $y$ and $\w[0]$ belong to $Y_{\GL_r}$, we have $\w_{0,R}(y') =Y_{\GL_r} \cap (n_\alpha Y_{\GL_R})=n_\alpha \cdot Y_{\GL_r}$. Moreover, it is clear that $y'=\widehat{\phi(y_1)}$. The proof is completed.
\end{proof}

\vskip 5pt
\section{Generalized Bump-Hoffstein conjecture}

\subsection{Whittaker functionals}
We follow \cite{Ga2} for notations in this section. Let $\wt{\GL}_r$ be an  $n$-fold covering group with Borel subgroup $\wt{B}=\wt{T} U$. A representation $\pi$ of $\wt{\GL}_r$ is called genuine if $\bbmu_n$ acts by $\iota: \bbmu_n \into \C^\times$, and it is called unramified if $\pi^K\ne 0$.

Consider the maximal abelian subgroup 
$$\wt{A}:=Z(\wt{T}) \mbf{T}(O_F)$$ 
of $\wt{T}$.
Let $\chi \in \text{Hom}_\iota (Z(\wt{T}), \C^\times)$ be a genuine unramified character, i.e., $\chi$ is trivial on $Z(\wt{T}) \cap K$. Let $i(\chi):=\text{Ind}_{\wt{A}}^{\wt{T}} \chi$ be the irreducible genuine representation of $\wt{T}$,  where $\chi$ also denotes the canonical extension to $\wt{A}$. Consider the normalized principal series  representation 
$$I(\chi):=\text{Ind}_{\wt{B}}^{\wt{\GL}_r} i(\chi)\otimes \mbf{1}.$$

Let $\Ftn(i(\chi))$ be the vector space of  functions $\cc$ on $\wt{T}$  satisfying
$$\cc(\wt{t} \cdot \wt{z}) =  \cc(\wt{t}) \cdot \chi(\wt{z}), \quad \wt{t} \in \wt{T} \text{ and } \wt{z} \in \wt{A}.$$
The support of  $\cc \in \Ftn(i(\chi))$ is a disjoint union of cosets in $\wt{T}/\wt{A}$. Moreover, 
\begin{equation} \label{F:dim}
\dim \Ftn(i(\chi))=\val{Y_{\GL_r}/Y_{\GL_r, Q,n}}
\end{equation}
since $\wt{T}/\wt{A}$ has the same size as $Y_{\GL_r}/Y_{\GL_r, Q,n}$. 
\vskip 5pt

There is a natural isomorphism of vector spaces $\Ftn(i(\chi)) \simeq i(\chi)^\vee$, where $i(\chi)^\vee$ is the complex dual space of functionals of $i(\chi)$. More explicitly, let $\set{\gamma_i}\subseteq \wt{T}$ be a chosen set of representatives of $\wt{T}/\wt{A}$, consider $\cc_{\gamma_i} \in \Ftn(i(\chi))$ which has support $\gamma_i \cdot \wt{A}$ and $\cc_{\gamma_i}(\gamma_i)=1$. It gives rise to a linear functional $\lambda_{\gamma_i}^{\chi} \in i(\chi)^\vee$ such that $\lambda_{\gamma_i}^{\chi}(f_{\gamma_j})=\delta_{ij}$, where $f_{\gamma_j}\in i(\chi)$ is the unique element such that $\text{supp}(f_{\gamma_j})=\wt{A}\cdot \gamma_j^{-1}$ and $f_{\gamma_j}(\gamma_j^{-1})=1$. That is, $f_{\gamma_j}=i(\chi)(\gamma_j)\phi_0$, where $\phi_0\in i(\chi)$ is the normalized unramified vector of $i(\chi)$ such that $\phi_0(1_{\wt{T}})=1$.  In view of this, the isomorphism $\Ftn(i(\chi)) \simeq i(\chi)^\vee$ is given explicitly by
\begin{equation} \label{F:cWh}
 \cc \mapsto   \lambda_\cc^{\chi}:= \sum_{\gamma_i \in \wt{T}/\wt{A}} \cc(\gamma_i) \lambda_{\gamma_i}^{\chi}.
\end{equation}
It can be checked easily that the isomorphism does not depend on the choice of representatives for $\wt{T}/\wt{A}$.

Fix an additive character $\psi: F\to \C^\times$ of conductor $O_F$, the ring of integers of $F$.  Also denote by $\psi: U\to \C^\times$ the character on $U$ such that its restriction to every $U_\alpha$ for $\alpha\in \Delta$ is given by $\psi \circ e_\alpha^{-1}$.

\begin{dfn}
For a genuine irreducible representation $(\pi, V_\pi)$ of $\wt{\GL}_r$, a linear functional $l: V_{\pi} \to \C$ is called a Whittaker functional if $l(\pi(u) v)= \psi(u) \cdot v$ for all $u\in U$ and $v\in V_{\pi}$.
Denote by $\Wh(\pi)$ the space of Whittaker functionals for $\pi$. A genuine representation $\pi$ is called $\psi$-generic (or simply generic) if $\dim \Wh(\pi) \ge 1$.
\end{dfn}

There is an isomorphism between $i(\chi)^\vee$ and the space $\Wh(I(\chi))$ of Whittaker functionals  on $I(i(\chi))$ (cf. \cite{Mc2}), given by $\lambda^\chi \mapsto l_\lambda$ with
$$l_\lambda:  I(\chi) \to \C, \quad f \mapsto \lambda^\chi \left( \int_{U^-} f(u w_0) \psi(u) \mu(u) \right),$$
where $f\in I(\chi)$ is an $i(\chi)$-valued function on $\wt{\GL}_r$. Here $U^-$ is the unipotent subgroup opposite to $U$. For any $\cc\in \Ftn(i(\chi))$, by abuse of notation, we will write $\lambda_\cc^{\chi} \in \Wh(I(\chi))$ for the resulting Whittaker functional of $I(\chi)$ from the composition of isomorphisms $\Ftn(i(\chi))\simeq i(\chi)^\vee \simeq \Wh(I(\chi))$. It is an immediate consequence of (\ref{F:dim}) that
$$\dim \Wh(I(\chi)) = \val{Y_{\GL_r}/Y_{\GL_r, Q,n}}.$$

\vskip 5pt
Let $T_{\w, \chi}: I(\chi) \to I({}^{\w}\chi), \w \in W_{\GL_r}$ be the standard intertwining operator for $\w \in W_{\GL_r}$, defined by analytic continuation of a certain integral (cf. \cite{Mc1}). The intertwining operator $T_{\w_\alpha, \chi}: I(\chi) \to I({}^{\w_\alpha}\chi)$ for any $\alpha\in \Delta_{\GL_r}$ is determined by (cf. \cite{Mc2} and \cite{Ga1})
$$T_{\w_\alpha, \chi} (v_0) = \gk(\w_\alpha, \chi) \cdot v_0' \text{ with } \gk(\w_\alpha, \chi)=\frac{1-q^{-1}\chi(\wt{h}_\alpha(\varpi^{n_\alpha}))}{1-\chi(\wt{h}_\alpha(\varpi^{n_\alpha})}, $$
where $v_0\in I(\chi)$ and $v_0' \in I({}^{\w_\alpha} \chi)$ are the normalized unramified vectors. For general $\w \in W_{\GL_r}$, denote
$$\Phi_\w:=\set{\alpha\in \Phi_{\GL_r}: \alpha>0 \text{ and } \w(\alpha)<0}.$$
Then the Gindikin-Karpelevich  coefficient $\gk(\w, \chi)$ associated with $T_{\w, \chi}$ is
$$\gk(\w, \chi)=\prod_{\alpha\in \Phi_\w} \gk(\w_\alpha, \chi)$$
such that $T_{\w, \chi} (v_0) = \gk(\w, \chi) \cdot v_0'$.

\vskip 5pt

In general, let $J(\w, \chi)$ be the image of $T_{\w, \chi}$. The operator $T_{\w, \chi}$ induces a homomorphism 
$$ T_{\w, \chi}^*: \Wh(I({}^{\w}\chi))\to  \Wh(I(\chi)),$$
whose image is $\Wh(J(\w, \chi))$. We note that $T_{\w,\chi}^*$ is given by 
$$\angb{\lambda_\cc^{{}^{\w}\chi} }{-} \mapsto \angb{\lambda_\cc^{{}^{\w}\chi} }{T_{\w,\chi}(-)}$$
for any $\cc \in \Ftn(i({}^{\w}\chi))$. Let $\set{\lambda_{\gamma}^{^{\w}\chi}}_{\gamma \in \wt{T}/\wt{A}}$ be a basis for  $\Wh(I({}^{\w}\chi))$, and $\set{ \lambda_{\gamma'}^{\chi} }_{\gamma'\in \wt{T}/\wt{A}}$ a basis for $\Wh(I(\chi))$. The map $T_{\w,\chi}^*$ is then determined by the matrix
$[\tau(\chi, \w, \gamma, \gamma')]_{\gamma, \gamma'\in \wt{T}/\wt{A}}$ of size $\val{Y_{\GL_r}/Y_{\GL_r,Q,n}}$
such that
$$T_{\w, \chi}^*(\lambda_{\gamma}^{^{\w}\chi}) = \sum_{\gamma'\in \wt{T}/\wt{A}} \tau(\chi, \w, \gamma, \gamma') \cdot \lambda_{\gamma'}^{\chi}.$$

\subsection{The Shahidi local coefficient matrix}
We will need some properties of the matrix $\left[\tau(\chi, \w, \gamma, \gamma')\right]_{\gamma, \gamma'}$ for an unramified character $\chi$.  The matrix is the analogue of Shahidi's local coefficient in the linear algebraic case, see \cite[Chapter 5]{Sha}.  We will recall the matrix for simple reflection $\w=\w_\alpha$ for $\alpha \in \Delta_{\GL_r}$.

Let $du$ be the normalized Haar measure of $F$ such that $du(O_F)=1$; thus, $du(O_F^\times)=1-1/q$. The Gauss sum is defined by
$$G_\psi(a, b)=\int_{O^\times_F} (u, \varpi)_n^a \cdot \psi(\varpi^b u) du, \quad a, b\in \Z.$$
It is known that
\begin{equation} \label{F:gauss}
G_\psi(a, b)=
\begin{cases}
0 & \text{ if } b<-1, \\
1-1/q & \text{ if } n| a \text{ and } b\ge 0,\\
0 &\text{ if } n\nmid a \text{ and } b\ge 0, \\
-1/q &\text{ if } n|a \text{ and } b=-1,\\
G_\psi(a, -1) \text{ with } |G_\psi(a,-1)|=q^{-1/2} &\text{ if } n\nmid a \text{ and } b=-1.
\end{cases}
\end{equation}
One has 
\begin{equation} \label{F:Gneg}
\overline{G_\psi(a, b)}= G_\psi(-a, b),
\end{equation}
where the left hand side denotes the conjugation of $G_\psi(a, b)$. For any $k\in \Z$, we write
$$\mathbf{g}_{\psi}(k):=G_{\psi}(k, -1).$$

For any real number $x$, we write $\ceil{x}\in \Z$ for the minimum integer such that $\ceil{x}\ge x$. To summarize properties of $[\tau(\chi, \w_\alpha, \gamma, \gamma')]_{\gamma, \gamma'}$, we state the following theorem by McNamara (with refinement from \cite{Ga2}) which generalizes \cite[Lemma I.3.3]{KP}.
\begin{thm}[{\cite[Theorem 13.1]{Mc2}}] \label{T:SLCM}
Suppose that $\gamma=\s_{y_1}$ is represented by $y_1$ and $\gamma'=\s_y$ by $y$. Then we can write $\tau(\chi, \w_\alpha, \gamma, \gamma')=\tau^1(\chi, \w_\alpha, \gamma, \gamma') + \tau^2(\chi, \w_\alpha, \gamma, \gamma')$ with the following properties:
\begin{enumerate}
\item[$\bullet$] $\tau^i(\chi, \w_\alpha,\gamma \cdot \wt{z}, \gamma' \cdot \wt{z}')=({}^{\w_\alpha} \chi)^{-1}(\wt{z}) \cdot \tau^i(\chi, \w_\alpha, \gamma, \gamma') \cdot \chi(\wt{z}'), \quad \wt{z}, \wt{z}'\in \wt{A}$;
\item[$\bullet$] $\tau^1(\chi, \w_\alpha, \gamma, \gamma')=0$  unless  $y_1 \equiv y \mod Y_{\GL_r,Q,n}$;
\item[$\bullet$] $\tau^2(\chi, \w_\alpha, \gamma, \gamma')=0$   unless $y_1 \equiv \w_\alpha[y] \mod Y_{\GL_r, Q,n}$.
\end{enumerate}
Moreover, 
\begin{enumerate}
\item[$\bullet$] If $y_1= y$, then 
$$\tau^1(\chi, \w_\alpha, \gamma, \gamma')=(1-q^{-1}) \frac{\chi (\wt{h}_\alpha(\varpi^{n_\alpha}))^{k_{y,\alpha}}}{1-\chi (\wt{h}_\alpha(\varpi^{n_\alpha}))}, \text{ where } k_{y,\alpha}=\ceil{\frac{\angb{y}{\alpha}}{n_\alpha}}.$$
\item[$\bullet$] If $y_1=\w_\alpha[y]$, then
$$\tau^2(\chi, \w_\alpha, \gamma, \gamma') =  \g(\angb{y_\rho}{\alpha}Q(\alpha^\vee)).$$
\end{enumerate}
\end{thm}

\subsection{Whittaker function}
Now for any $\lambda^\chi \in i(\chi)^\vee$, we define the Whittaker function $\mca{W}_\lambda^{\GL_r}$ by
$$\mca{W}_\lambda^{\GL_r}(g):=\lambda^\chi ( I(\chi)(g)v_0 ),$$
where $v_0 \in I(\chi)$ is the normalized unramified vector. We also write
$$\mca{W}_\cc^{\GL_r}:=\mca{W}_{\lambda_\cc}^{\GL_r} \text{ for } \cc \in \Ftn(i(\chi)), \text{ and } \mca{W}_\gamma^{\GL_r}:=\mca{W}_{\lambda_{\cc_\gamma}}^{\GL_r} \text{ for } \gamma\in \wt{T}$$
for the arising Whittaker models for $I(\chi)$.

For any lifting $\wt{t}\in \wt{T}$ of $t\in T$, write
$$\delta_{B}^{1/2}(\wt{t}):=\delta_{B}^{1/2}(t),$$
where $\delta_B$ is the modular character of $B$. An element $\wt{t}\in \wt{T}$ is called dominant if $\wt{t} \cdot (U\cap K) \cdot \wt{t}^{-1} \subseteq K$.  The following result generalizes \cite[Theorem 4.2]{Pat} for Kazhdan-Patterson coverings. In fact, the proof shows that it holds for coverings of general reductive groups besides $\wt{\GL}_r$ as well.

\begin{prop} \label{P:Pat}
Consider an $n$-fold covering group $\wt{\GL}_r$. Let $I(\chi)$ be an unramified principal series  of $\wt{\GL}_r$ and $\gamma \in \wt{T}$. Let $\mca{W}_\gamma^{\GL_r}$ be the Whittaker model of $I(\chi)$ associated to $\gamma$. Then, $\mca{W}_\gamma^{\GL_r}(\wt{t})=0$ unless $\wt{t}$ is dominant. Moreover, for dominant $\wt{t}\in \wt{T}$, we have
$$ \mca{W}_\gamma^{\GL_r}(\wt{t})=\delta_B^{1/2}( \wt{t} ) \cdot \sum_{\w \in W_{\GL_r}} \gk(\w_{0,r} \w^{-1}, \chi)  \tau(\w, {}^{\w^{-1}}\chi, \gamma, w_{0,r} \cdot \wt{t} \cdot w_{0,r}^{-1}) .$$
\end{prop}
\begin{proof}
We will change the notation slightly only for the proof: denote by $\mca{W}_\gamma^{\chi}$ the Whittaker function of $I(\chi)$ associated with $\gamma \in \wt{T}$. For convenience, we also introduce the temporary notation
$$\breve{T}_{\w, \chi}:=\gk(\w, \chi)^{-1} \cdot T_{\w, \chi}, \text{ and } \breve{\tau}(\chi, \w, \gamma, \gamma'):=\gk(\w, \chi)^{-1} \cdot \tau(\chi, \w, \gamma, \gamma').$$
By the definition of $\tau(\chi, \w, \gamma, \gamma')$, one has
$$\lambda_\gamma^{\chi}\circ \breve{T}_{\w, \wchi} = \sum_{\gamma'\in \wt{T}/\wt{A}} \breve{\tau}({}^{\w^{-1}}\chi, \w, \gamma, \gamma') \cdot \lambda_{\gamma'}^{{}^{\w^{-1}}\chi}.$$

Write $\w_0:=\w_{0,r}$. As in \cite[\S 5]{Mc2}, let $\phi_{\w_0}^{\chi} \in I(\chi)$ be the element supported on $\wt{B}w_0 I$ and $\phi_{\w_0}^{\chi}(w_0)=\phi_0$ is the unramified vector in $i(\chi)$, where $I \subset K$ is the standard Iwahori subgroup associated to $\Delta_{\GL_r}$. Denote by $R(\wt{g})$ the right translation action of $\wt{g}$ on the space of $I(\chi)$. Then, by \cite[Theorem 8.1]{Mc2}, it suffices to compute for dominant $\wt{t} \in \wt{T}$ that
\begin{align*}
& \mca{W}_\gamma^{\chi}(\wt{t}) \\
=& \sum_{\w \in W} \gk(\w_0, {}^{\w^{-1} }\chi) \cdot \mca{W}_\gamma^{\chi}\left( \breve{T}_{\w, {}^{\w^{-1}}\chi} \big(R(\wt{t})\phi_{\w_0}^{ {}^{\w^{-1}}\chi  } \big) \right) \\
=&  \sum_{\w \in W} \gk(\w_0, {}^{\w^{-1} }\chi) \sum_{\gamma'\in \wt{T}/\wt{A}} \breve{\tau}(\w, {}^{\w^{-1}}\chi, \gamma, \gamma') \cdot \mca{W}_{\gamma'}^{ {}^{\w^{-1}}\chi } \left( R(\wt{t})\phi_{\w_0}^{ {}^{\w^{-1}}\chi  } \right) \\
=& \sum_{\w \in W} \gk(\w_0, {}^{\w^{-1} }\chi) \sum_{\gamma'\in \wt{T}/\wt{A}} \breve{\tau}(\w, {}^{\w^{-1}}\chi, \gamma, \gamma') \cdot \delta_B^{1/2}(t) \cdot \lambda_{\gamma'}^{ {}^{\w^{-1}}\chi } \left( i({}^{\w^{-1}}\chi)(w_0\cdot \wt{t} \cdot w_0^{-1}) \phi_0 \right),
\end{align*} 
where the last equality follows from \cite[Lemma 6.3]{Mc2}. However, 
$$\lambda_{\gamma'}^{ {}^{\w^{-1}}\chi } \left( i({}^{\w^{-1}}\chi)(w_0\cdot \wt{t} \cdot w_0^{-1}) \phi_0 \right)\ne 0$$
if and only if 
$$\gamma' \in w_0 \wt{t} w_0^{-1} \cdot \wt{A}.$$
Therefore,
$$ \mca{W}_\gamma^{\chi}(\wt{t})=\delta_B^{1/2}(t) \cdot \sum_{\w \in W} \gk(\w_0, {}^{\w^{-1} }\chi)\cdot  \breve{\tau}(\w, {}^{\w^{-1}}\chi, \gamma, w_0 \cdot \wt{t} \cdot w_0^{-1}) .$$
The result follows from the cocycle relation of the Gindikin-Karpelevich coefficients.
\end{proof}

\subsection{Theta representation and its Whittaker functions}
An unramified genuine character $\chi$ of $Z(\wt{T}) \subseteq \wt{\GL}_r$ is called exceptional if $$\chi(\wt{h}_\alpha(\varpi^{n_\alpha}))=q^{-1}$$
 for all $\alpha\in \Delta_{\GL_r}$. The theta representation $\Theta(\wt{\GL}_r, \chi)$ associated to an exceptional character $\chi$ is the unique Langlands quotient (cf. \cite{BJ}) of $I(\chi)$, which is also equal to the image of the intertwining operator $T_{\w_{0,r}, \chi}: I(\chi) \to I({}^{\w_{0,r}}\chi)$. Therefore, we could identify $\Wh (\Theta(\wt{\GL}_r, \chi))$ as a subspace of $\Wh (I(\chi))$. Following \cite{Ga2}, we call $\Theta(\wt{\GL}_r, \chi)$ distinguished if $\dim \Wh (\Theta(\wt{\GL}_r, \chi))=1$.

\begin{prop} \label{P:Wh-Th}
Assume that $\cc \in \Ftn(i(\chi))$ gives rise to a normalized Whittaker model $\mca{W}_\cc^{\GL_r}$ of $\Theta(\wt{\GL}_r, \chi)$. Then for any dominant $\wt{t}$, one has
$$\mca{W}_\cc^{\GL_r}(\wt{t})=\gk(\w_{0,r}, \chi) \cdot \delta_B^{1/2}(\wt{t}) \cdot \cc (w_{0,r} \cdot \wt{t} \cdot w_{0,r}^{-1}). $$
If $\cc(1_{\wt{T}})=1$, then $\mca{W}_\cc^{\GL_r}(\wt{t})= \delta_B^{1/2}(\wt{t}) \cdot \cc (w_{0,r} \cdot \wt{t} \cdot w_{0,r}^{-1}) \cdot \mca{W}_\cc^{\GL_r}(1)$; in particular,
$$\mca{W}_\cc^{\GL_r}(\s_y)= \delta_B^{1/2}(\s_y) \cdot \cc (\s_{\w_{0,r}(y)}) \cdot \mca{W}_\cc^{\GL_r}(1)$$
for any $\s_y$ dominant. 
\end{prop}
\begin{proof}
It follows from (\ref{F:cWh}) that for any $\cc \in \Ftn(i(\chi))$ we have
\begin{align*}
\mca{W}_\cc^{\GL_r}(\wt{t}) = & \sum_{\gamma \in \wt{T}/\wt{A}} \cc(\gamma) \cdot \mca{W}_\gamma^{\GL_r}( \wt{t} ) \\
=& \delta_B^{1/2}(\wt{t}) \cdot \sum_{\w \in W_{\GL_r}} \gk(\w_{0,r} \w^{-1}, \chi) \left(  \sum_{\gamma\in \wt{T}/\wt{A}} \cc(\gamma)\cdot \tau(\w, {}^{\w^{-1}}\chi, \gamma, w_{0,r} \cdot \wt{t} \cdot w_{0,r}^{-1}) \right).
\end{align*}

We first show that for any $\w\in W_{\GL_r}$ with $\w\ne \id$, one has
$$\sum_{\gamma\in \wt{T}/\wt{A}} \cc(\gamma)\cdot \tau(\w, {}^{\w^{-1}}\chi, \gamma, \gamma')=0 \text{ for any } \gamma' \in \wt{T}.$$
For $\w=\w_\alpha$ with $\alpha\in \Delta_{\GL_r}$, this is just the equality in \cite[Corollary 3.5]{Ga2}. In general, write $\w=\w_\alpha \w_1$ with $\alpha\in \Delta_{\GL_r}$ and $l(\w_1)=l(\w)-1$. Then the cocycle relation in \cite[Lemma 3.2]{Ga2} yields
\begin{align*}
& \sum_{\gamma\in \wt{T}/\wt{A}} \cc(\gamma)\cdot \tau(\w, {}^{\w^{-1}}\chi, \gamma, \gamma') \\
=& \sum_{\gamma\in \wt{T}/\wt{A}} \cc(\gamma) \left( \sum_{\gamma'' \in \wt{T}/\wt{A}} \tau(\w_\alpha, {}^{\w_\alpha^{-1}}\chi, \gamma, \gamma'') \cdot \tau(\w_1, {}^{\w^{-1}}\chi, \gamma'', \gamma') \right) \\
=& \sum_{\gamma'' \in \wt{T}/\wt{A}}  \left( \sum_{\gamma \in \wt{T}/\wt{A}} \cc(\gamma) \tau(\w_\alpha, {}^{\w_\alpha^{-1}}\chi, \gamma, \gamma'')  \right) \cdot \tau(\w_1, {}^{\w^{-1}}\chi, \gamma'', \gamma') \\
=&0,
\end{align*}
where the last equality again follows from \cite[Corollary 3.5]{Ga2}. Therefore
$$\mca{W}_\cc^{\GL_r}(\wt{t})=\delta_B^{1/2}( \wt{t} ) \cdot \gk(\w_{0,r}, \chi) \left(  \sum_{\gamma\in \wt{T}/\wt{A}} \cc(\gamma)\cdot \tau(\id, \chi, \gamma, w_{0,r} \cdot \wt{t} \cdot w_{0,r}^{-1}) \right).$$
Since $\tau(\id, \chi, \gamma, \gamma')=0$ if $\gamma\notin \gamma'\cdot \wt{A}$, the result follows. The last claim regarding the case $\wt{t}=\s_y$ follows from Lemma \ref{L:2n}. This completes the proof.
\end{proof}

More importantly, the dimension of Whittaker functionals for a theta representation is determined as follows.
\begin{prop} \label{P:dim}
For a covering group $\wt{\GL}_r^{(n)}$, the equality $Y_{\GL_r, Q,n}^{sc}= Y_{\GL_r, Q,n} \cap Y^{sc}_{\GL_r}$ always holds. It follows that $\OF_{Q,n} =\OF_{Q,n,sc}$ and therefore
\begin{equation} \label{F:dim}
\dim \Wh (\Theta(\wt{\GL}_r^{(n)}, \chi))=\val{ \wp(\OF_{Q,n}) }=\val{ \wp(\OF_{Q,n,sc}) }.
\end{equation}
If $n_\alpha \ge r$, then the Weyl-orbit $\mca{O}_0 \subset Y_{\GL_r}$ lies in $\OF_{Q,n}$ and therefore $\Theta(\wt{\GL}_r^{(n)}, \chi)$ is generic.
Moreover, if $Y_{\GL_r, Q, n}=n_\alpha \cdot Y_{\GL_r}$, then
$$\dim \Wh (\Theta(\wt{\GL}_r^{(n)}, \chi))=\val{ \wp(\OF_{Q,n}) }=\binom{n_\alpha}{r}.$$
\end{prop}
\begin{proof} For simplicity, we omit the subscript $\GL_r$ for the notations in the proof. It is clear that $Y_{Q,n}^{sc} \subset Y_{Q,n} \cap Y^{sc}$. Conversely, $y=\sum_i k_i e_i$ belongs to $Y^{sc}$ if and only if $\sum_i k_i=0$. If furthermore $y\in Y_{Q,n}$, it follows from (\ref{YQn}) that $n_\alpha |k_i$ for all $i$. That is, $y$ lies in $n_\alpha \cdot Y^{sc}$, which is exactly $Y_{Q,n}^{sc}$.

Now we show $\OF_{Q,n} =\OF_{Q,n,sc}$. Clearly, $\OF_{Q,n} \subseteq \OF_{Q,n, sc}$. On the other hand, if $\mca{O}_y$ is not $Y_{Q,n}$-free, then there exists $z \in \mca{O}_y$ and $\w \in W_{\GL_r}$ with $\w\ne \text{id}$ such that $\w [z]-z \in Y_{Q,n}$. Since $Y_{Q,n}\cap Y^{sc} =Y_{Q,n}^{sc}$, we see that $\mca{O}_y$ is not $Y_{Q,n}^{sc}$-free as well.

The equality (\ref{F:dim}) follows from \cite[Theorem 3.14]{Ga2}, since $\OF_{Q,n}=\OF_{Q,n,sc}$ in this case. Moreover, the last two assertions including $\val{ \wp(\OF_{Q,n}) }=\binom{n_\alpha}{r}$ follow from a simple combinatorial computation.
\end{proof}

\begin{cor} Let $(\wt{\GL}_r^{(n)}, \wt{\GL}_R^{(n)})$ be a fundamental pair. Then the theta representation $\Theta(\wt{\GL}_R^{(n)}, \chi)$ for any exceptional $\chi$ for $\wt{\GL}_R^{(n)}$ is distinguished, i.e., it has uniqueness of Whittaker functionals. Moreover, $\Theta(\wt{\GL}_r^{(n)}, \mu)$ is generic in this case.
\end{cor}
\begin{proof}
In view of Proposition \ref{P:dim}, it follows from the condition $n_\alpha=R >r$ in (\textsf{FP2}) for a fundamental pair, see Definition \ref{D:FP}.
\end{proof}

\begin{rmk}
Our convention of Whittaker functions agrees with \cite{Mc1}, \cite{Mc2}, \cite{Suz1} and \cite{Ga2}, and is slightly different from that in \cite{KP} and \cite{Suz3}. One main difference is that in \cite{KP} and \cite{Suz3}, the authors do not fix the central character of $\wt{\GL}_r^{(n)}$, and therefore for them Whittaker functionals are considered simultaneously  for all possible central characters, see \cite[\S I.3]{KP} and \cite[\S 3.1]{Suz3} for details. This explains the difference of formulas for the dimension of Whittaker functionals as in \cite[pg.735]{Suz3} compared to Proposition \ref{P:dim} above.
\end{rmk}

\vskip 5pt
Now we give an explicit description of elements of $\Wh (\Theta(\wt{\GL}_r^{(n)}, \chi))$.  Recall that in \cite{Ga2}, for any $y\in Y_{\GL_r}$ and $\alpha\in \Delta_{\GL_r}$, we write
$$\mathbf{t}_{\GL_r}(\w_\alpha, y):=q^{k_{y,\alpha}-1} \cdot \g(\angb{y_\rho}{\alpha}Q(\alpha^\vee))^{-1},$$
where
$$k_{y,\alpha}=\ceil{\frac{\angb{y}{\alpha}}{n_\alpha}}.$$
Consider any $Y_{\GL_r, Q,n}$-free orbit $\mca{O}_y \in \OF_{Q,n}$. For $\w=\w_k ...\w_2 \w_1\in W_{\GL_r}$ in a minimum decomposition, consider
$$\mathbf{t}_{\GL_r}(\w, y):=\prod_{i=1}^k \mathbf{t}_{\GL_r}(\w_i, \w_{i-1}...\w_1[y]),$$
which is well-defined and independent of the minimum expansion of $\w$ (see \cite[Proposition 3.10]{Ga2}).

Any $\mca{O}_y \in \OF_{Q,n}$ gives rise to an element $\cc_{\mca{O}_y} \in \Ftn(i(\chi))$ as follows (cf. \cite[\S 3.4]{Ga2}):
\begin{enumerate}
\item[$\bullet$] First, let $\cc_{\mca{O}_y}(\s_y)=1$. For any $\alpha\in \Delta_{\GL_r}$, define
$$\cc_{\mca{O}_y}(\s_{\w_\alpha[y]}):= \mathbf{t}_{\GL_r}(\w_\alpha, y) \cdot \cc_{\mca{O}_y}(\s_y)=\mathbf{t}_{\GL_r}(\w_\alpha, y).$$
Inductively, one can define 
\begin{equation} \label{F:cc}
\cc_{\mca{O}_y}(\s_{\w[y]}):=\mathbf{t}_{\GL_r}(\w, y)
\end{equation}
for any $\w\in W_{\GL_r}$. It is well-defined and independent of the minimum decomposition of $\w$. 
\item[$\bullet$] Second, extend $\cc_{\mca{O}_y}$ to a function on $\wt{T}$ by
$$\cc_{\mca{O}_y}(\s_{\w[y]} \cdot \wt{z}) =  \cc_{\mca{O}_y}(\s_{\w[y]}) \cdot \chi(\wt{z}), \ \wt{z} \in \wt{A}.$$
and
$$\cc_{\mca{O}_y}(\wt{t})=0 \text{ if } \wt{t} \notin \bigcup_{\w\in W_{\GL_r}} \s_{\w[y]}  \cdot \wt{A}.$$
\end{enumerate}
Following this, denote
$$\mca{W}_{\mca{O}_y}^{\GL_r}:=\mca{W}_{\cc_{\mca{O}_y}}^{\GL_r} \in  \Wh (\Theta(\wt{\GL}_r^{(n)}, \chi)).$$  

Let $\set{\mca{O}_{y_i} \in \OF_{Q,n}: 1\le i\le \val{\wp(\OF_{Q,n})} }$
be a set of representatives for $\wp(\OF_{Q,n})$. The set $\set{\cc_{\mca{O}_{y_i}}}_i$ gives rises to a basis $\set{\mca{W}^{\GL_r}_{\mca{O}_{y_i}}}_i$ for $\Wh(\Theta(\wt{\GL}_r^{(n)}, \chi))$. Assume that $n_\alpha\ge r$, we will concentrate on $\mca{O}_0 \in \OF_{Q,n}$, and consider properties of $\cc_{\mca{O}_0}$, which determines $\mca{W}^{\GL_r}_{\mca{O}_0}$ by Proposition \ref{P:Wh-Th}.
\vskip 5pt

The following result is an analogue of \cite[Proposition 4]{Suz1} for Kazhdan-Patterson coverings.
\begin{prop} \label{P:cc-val}
Assume that $n_\alpha \ge r$. Let $\w\in W_{\GL_r}$ be any element. Then $\cc_{\mca{O}_0}(\s_{\w[0]}) = \cc_{\mca{O}_0}(\s_{\widehat{\w}[0]})$ with $\val{\cc_{\mca{O}_0}(\s_{\w[0]})}=q^{-l(\w)/2}$.
\end{prop}
\begin{proof}
By definition (\ref{F:cc}), we have $\cc_{\mca{O}_0}(\s_{\w[0]})=\mathbf{t}(\w, 0)$ for any $\w\in W_{\GL_r}$. Thus, it suffices to show the statement for $\mathbf{t}(\w,0)$, and we do this by induction.

The statement is clearly true for $\w=\text{id}$. Let $\w=\w_\alpha \w_1$ be a decomposition with $\alpha\in \Delta_{\GL_r}$ and $l(\w)= 1 + l(\w_1)$. We assume that the result holds for $\w_1\in W_{\GL_r}$, i.e. $\mathbf{t}(\w_1,0) = \mathbf{t}(\widehat{\w_1}, 0)$ with absolute value $q^{-l(\w_1)/2}$.

We first show 
$$\mathbf{t}_{\GL_r}(\w_\alpha, \w_1[0])=\mathbf{t}_{\GL_r}(\widehat{\w_\alpha}, \widehat{\w_1}[0]).$$
Recall that for any $\alpha\in \Delta_{\GL_r}$ and $y\in Y_{\GL_r}$, one has
$$\mathbf{t}_{\GL_r}(\w_\alpha, y)=q^{k_{y,\alpha}-1} \cdot \g(\angb{y_\rho}{\alpha}Q(\alpha^\vee))^{-1},$$
where
$$k_{y,\alpha}=\ceil{\frac{\angb{y}{\alpha}}{n_\alpha}}.$$

By Lemma \ref{L:01}, we have $\widehat{\w_\alpha}=\w_{\beta}$ with $\beta^\vee=-\w_{0,r}(\alpha^\vee)$. Therefore, we show the equality $\mathbf{t}_{\GL_r}(\w_\alpha, \w_1[0])=\mathbf{t}_{\GL_r}(\widehat{\w_\alpha}, \widehat{\w_1}[0])$ in two steps:
\begin{enumerate}
\item[$\bullet$] First, we have
\begin{align*}
& \angb{ \widehat{\w_1}[0]  }{\beta} \\
=& \angb{ \w_{0,r} \w_1 \w_{0,r}^{-1}[0]  }{ -\w_{0,r}(\alpha)} \\
=& \angb{ \w_{0,r} \w_1 [2\rho_{\GL_r}]  }{ -\w_{0,r}(\alpha)} \\
=& - \angb{ \w_{0,r} \w_1(\rho_{\GL_r}) + \rho_{\GL_r}  }{ \w_{0,r}(\alpha)} \\
=& - \angb{ \w_0 \w_0 \w_1(\rho_{\GL_r}) + \w_0(\rho_{\GL_r})  }{ \alpha} \\
=& \angb{ \w_1(0-\rho_{\GL_r}) + \rho_{\GL_r}  }{ \alpha} \\
=& \angb{ \w_1[0]  }{\alpha}.
\end{align*}
This gives $k_{\w_1[0], \alpha}= k_{\widehat{\w_1}[0], \beta}$.
\item[$\bullet$] Second, it can be checked easily that
$$\angb{ \rho_{\GL_r} }{ \beta } = -\angb{ \w_0(\rho_{\GL_r})  }{\alpha}=\angb{\rho_{\GL_r} }{\alpha}.$$
It follows $\angb{ \w_1[0]_{\rho} }{ \alpha } = \angb{ \widehat{\w_1}[0]_\rho }{ \beta }$ and therefore
$$\g(\angb{ \w_1[0]_{\rho} }{ \alpha }Q(\alpha^\vee))=\g(\angb{ \widehat{\w_1}[0]_\rho }{ \beta }Q(\beta^\vee)),$$
since $Q(\alpha^\vee)=Q(\beta^\vee)$.
\end{enumerate}
This completes the proof that $\mathbf{t}_{\GL_r}(\w_\alpha, \w_1[0])=\mathbf{t}_{\GL_r}(\widehat{\w_\alpha}, \widehat{\w_1}[0])$.

\vskip 5pt
It remains to show $\val{ \mathbf{t}_{\GL_r}(\w_\alpha, \w_1[0]) }=q^{-1/2}$. For this, by induction, it suffices to show that
$$\ceil{ \frac{ \angb{\w_1[0] }{\alpha} }{n_\alpha} }=0 \text{ and } \val{ \g( \angb{\w_1[0]_\rho}{\alpha} \cdot Q(\alpha^\vee)) }= q^{-1/2}.$$
Since $\w_1[0]=\rho - \w_1(\rho)$, one has 
$$\angb{\w_1[0]}{\alpha} = 1- \angb{\rho}{\w_1^{-1}(\alpha)}.$$
Now $l(\w_\alpha \cdot \w_1)=1 + l(\w_1)$ implies  that $\w_1^{-1}(\alpha) \in \Phi_{\GL_r}^+$ (cf. \cite[p. 170, Corollary 2]{Bou}). Thus,
$$\angb{\rho}{\w_1^{-1}(\alpha)} \in [1, r-1].$$
It follows from $r\le n_\alpha$ that
$$\angb{\w_1[0]}{\alpha} \in [-(n_\alpha -2), 0]$$
and therefore
$$\ceil{ \frac{ \angb{\w_1[0] }{\alpha} }{n_\alpha} }=0.$$
Moreover, since $\angb{\w_1[0]_\rho}{\alpha} \in [-(n_\alpha -1), -1]$, it follows from (\ref{F:gauss}) that 
$$\val{ \g( \angb{\w_1[0]_\rho}{\alpha} \cdot Q(\alpha^\vee)) }= q^{-1/2}.$$
Therefore, $\val{ \mathbf{t}_{\GL_r} (\w_\alpha, \w_1[0]) }=q^{-1/2}$. The proof  is completed.
\end{proof}

Now we consider a fundamental pair $(\wt{\GL}_r^{(n)}, \wt{\GL}_R^{(n)})$. To differentiate the groups under consideration, we use subscript or superscript for the notations, for example, $\wp_{\GL_R}$ and $\OF_{\GL_R, Q,n}$ etc. By Proposition \ref{P:dim}, we have
$$\wp_{\GL_R}({\OF_{\GL_R, Q,n}}) = \set{\wp_{\GL_R} (\mca{O}_0^{\GL_R})},$$
where $\mca{O}_0^{\GL_R}$ is the $W_{\GL_R}$-orbit of $0$, which is $Y_{\GL_R, Q,n}$-free in this case. It gives rise to the unique Whittaker model $\mca{W}_{\mca{O}_0}^{\GL_R}$ of $\Theta(\wt{\GL}_R, \chi)$ associated to $\cc_{\mca{O}_0}^{\GL_R}$.

\begin{cor} \label{C:key1}
For any $\w\in W_{\GL_r}$, one has
$$\cc^{\GL_r}_{\mca{O}_0}(\s_{\widehat{\w}[0]})=\cc^{\GL_R}_{\mca{O}_0}(\s_{\widehat{\phi(\w)}[0]}) \text{ and } \val{ \cc^{\GL_r}_{\mca{O}_0}(\s_{\widehat{\w}[0]}) }=q^{-l(\w)/2}.$$
In particular, $\overline{\cc^{\GL_R}_{\mca{O}_0}(\s_{\widehat{\phi(\w)}[0]})}\cdot \cc^{\GL_r}_{\mca{O}_0}(\s_{\widehat{\w}[0]})=q^{-l(\w)}$, where $\overline{f}$ denotes the complex conjugate of any complex-valued function $f$.
\end{cor}
\begin{proof}
By Proposition \ref{P:cc-val}, it suffices to show
$$\cc^{\GL_r}_{\mca{O}_0}(\s_{\w[0]})=\cc^{\GL_R}_{\mca{O}_0}(\s_{\phi(\w)[0])}).$$
For this, one is reduced to show $\w[0] = \phi(\w)[0]$, which clearly holds as noted in \S \ref{S:W-act}.
\end{proof}

\subsection{Generalized Bump-Hoffstein conjecture} 
In this subsection, we first consider a covering roup $\wt{\GL}_r^{(n)}$ satisfying both
\begin{equation} \label{F:C1}
 n_\alpha\cdot  Y_{\GL_r} \subset Y_{\GL_r, Q,n}
\end{equation}
and 
\begin{equation} \label{F:C2}
n|(\p \cdot n_\alpha).
\end{equation}
With (\ref{F:C1}), we have the inclusions $Y_{\GL_r,Q,n}^{sc} \subseteq n_\alpha\cdot  Y_{\GL_r} \subseteq Y_{\GL_r, Q,n}$. Denote by $\mbf{T}_{Q,n, sc}$, $\mbf{T}_{n_\alpha}$ and $\mbf{T}_{Q,n}$ respectively the split torus over $F$ associated to these lattices. We also denote by $T_{Q,n, sc}$, $T_{n_\alpha}$ and $T_{Q,n}$ their $F$-rational points.

Note that $T=Y_{\GL_r}\otimes F^\times$. Let $T_{Q,n}^\dag$ be the image of the isogeny $T_{Q,n} \to T$, induced from the injection $Y_{\GL_r, Q,n} \to Y_{\GL_r}$. It is known that $Z(\wt{T}) \subset \wt{T}$ is the preimage of $T_{Q,n}^\dag$ with respect to the quotient map $\wt{T} \to T$ (cf. \cite{We1}). We also denote by $T_{n_\alpha}^\dag \subset T$ the image of $T_{n_\alpha} \to T$, induced from the embedding $n_\alpha Y_{\GL_r} \to Y_{\GL_r}$. 

\begin{lm} \label{L:split}
Assume that $\wt{\GL}_r^{(n)}$ satisfies (\ref{F:C1}) and (\ref{F:C2}). Then
$$D(y, z) \equiv D(z, y) \equiv 0 \mod n$$
for any $y\in n_\alpha Y_{\GL_r}$ and $z\in Y_{\GL_r}$. Therefore, the map given by
$$\bbmu_n \times T_{n_\alpha}^\dag \simeq \wt{T_{n_\alpha}^\dag}, \quad (\zeta, a^{n_\alpha y}) \mapsto \zeta \cdot \s(a^{n_\alpha y})$$
for any $a\in F^\times$ and $y\in Y_{\GL_r}$ is a group isomorphism.
\end{lm}
\begin{proof}
To prove the first assertion, by the definition of $D$ in (\ref{F:D}), it suffices to show that the numbers
$$n_\alpha\cdot D(e_i, e_i) \text{ and } n_\alpha\cdot D(e_i, e_j) \text{ for } i\ne j$$
are divisible by $n$. Note that $n| (2\p-\bq)\cdot n_\alpha$ always holds. Since $n|(\p\cdot n_\alpha)$, we have $n|(\bq\cdot n_\alpha)$ as well. The first assertion follows.

Now it follows from (\ref{F:s}) that 
$$\s(a^y)\cdot \s(b^z)= (a, b)_n^{D(y,z )}\cdot \s(a^y \cdot b^z)=\s(a^y \cdot b^z),$$
where $y, z\in n_\alpha\cdot Y_{\GL_r}$. Therefore, the second statement also holds.
\end{proof}

In fact, we have similar $T_{Q,n, sc}^\dag \subset T$ associated with $T_{Q,n,sc}$. By \cite{GG}, we have for general $\wt{\GL}_r^{(n)}$ (which does not necessarily satisfy (\ref{F:C1}) and (\ref{F:C2}) above) an isomorphism
$${\bbmu}_n \times T_{Q,n, sc}^\dag \simeq \wt{ T_{Q,n,sc}^\dag }$$
which is given by
$$(\zeta, \alpha^\vee(a^{n_\alpha})) \mapsto \zeta \cdot \wt{h}_\alpha(a^{n_\alpha}).$$
\vskip 5pt

In any case, we have $\wt{T_{n_\alpha}^\dag} \subset Z(\wt{T})$. For a genuine character $\chi: Z(\wt{T}) \to \C^\times$, we consider the linear character
$$\chi_{n_\alpha}: T_{n_\alpha} \to \C^\times$$
given by
\begin{equation} \label{F:lin}
\chi_{n_\alpha}(a^{n_\alpha y}):=\chi(\s(a^{n_\alpha y})) \text{ for all } y\in Y_{\GL_r}.
\end{equation}
By Lemma \ref{L:split}, $\chi_{n_\alpha}$ is well-defined.

\begin{lm} \label{L:exp-ch}
Assume that $\wt{\GL}_r^{(n)}$ satisfies (\ref{F:C1}) and (\ref{F:C2}). If $\chi \in \Hom(Z(\wt{T}), \C^\times)$ is an unramified exceptional character (in particular, $\chi(\wt{h}_\alpha(\varpi^{n_\alpha}))=q^{-1}$ for all $\alpha \in \Delta_{\GL_r}$), then we have
$$\chi_{n_\alpha}=\bigotimes_{i=1}^r |\cdot |^{\frac{r+1}{2} + \nu -i},$$
where $\nu \in \C$ is some complex number. Here $|\cdot |^s: (n_\alpha \Z e_i)\otimes F^\times \to \C^\times$ is the character sending $(n_\alpha e_i)(a)$ to $|a|^s$. 
\end{lm}
\begin{proof}
Let $\alpha^\vee=e_i - e_{i+1} \in \Delta_{\GL_r}^\vee$ be a simple coroot for some $1\le i \le r-1$. For any $a\in F^\times$, we have
\begin{align*}
& \wt{h}_\alpha(a^{n_\alpha}) \cdot \s(a^{n_\alpha e_{i+1}}) \\
=& \s(a^{n_\alpha \alpha^\vee}) \cdot \s(a^{n_\alpha e_{i+1}}) \\
=& \s(a^{n_\alpha e_i}) \cdot (a, a)_n^{D(n_\alpha \alpha^\vee, n_\alpha e_{i+1})} \\
=& \s(a^{n_\alpha e_i}).
\end{align*}
Therefore, if $\chi$ is exceptional, we have 
$$\chi_{n_\alpha}(a^{n_\alpha \alpha^\vee})=|a|$$
for all $\alpha \in \Delta_{\GL_r}$.
As $\chi_{n_\alpha}$ is unramified, the statement follows.
\end{proof}

By abuse of notation, we also denote by $\nu$ the character
$$\nu: (n_\alpha \Z e)\otimes F^\times \to \C^\times$$
which sends $(n_\alpha e)(a)$ to $|a|^{\nu}$.

\vskip 10pt
From now, we will consider a fundamental pair of $n$-fold covers $(\wt{\GL}_r, \wt{\GL}_R)$ satisfying $n|(\p\cdot n_\alpha)$. Denote by
$$\phi: \wt{\GL}_r \into \wt{\GL}_R$$
the natural inclusion.

Let $\pi$ be a generic unramified representation of $\wt{\GL}_r$. Let $\mu$ be an unramified character of $Z(\wt{T}_r) \subset \wt{\GL}_r$ such that $I(\mu) \onto \pi$. For $\wt{\GL}_R$, let $\chi$ be an exceptional character of $Z(\wt{T}_R) \subset \wt{\GL}_R$. We may write
$$\chi_{n_\alpha}[\nu_R], \nu_R \in \C$$
for the linear character $\chi_{n_\alpha}$ given in Lemma \ref{L:exp-ch}.
\vskip 5pt

Let $\mca{W}^{\GL_r}$ be any Whittaker model of $\pi$, and $\mca{W}^{\GL_R}_{\mca{O}_0}$ the unique Whittaker model of the distinguished theta representation $\Theta(\wt{\GL}_R, \chi)$. Consider the Rankin-Selberg integral
$$\msc{Z}(s, \overline{\mca{W}_{\mca{O}_0}^{\GL_R}} \times \mca{W}^{\GL_r})=\int_{U_r \backslash \GL_r} \overline{\mca{W}_{\mca{O}_0}^{\GL_R}}( \phi(\wt{g})) \cdot \mca{W}^{\GL_r}(\wt{g}) \cdot  |\det(g)|^{s-\frac{R-r}{2}} dg,$$
where $\overline{\mca{W}_{\mca{O}_0}^{\GL_R}}$ is the complex conjugate of $\mca{W}_{\mca{O}_0}^{\GL_R}$. Note that the integrand is a well-defined function on $U_r\backslash \GL_r$.

\vskip 5pt
\begin{conj}[Bump-Hoffstein] \label{C:GBH}
Let $(\wt{\GL}_r, \wt{\GL}_R)$ be a fundamental pair of $n$-fold covering groups satisfying $n|(\p\cdot n_\alpha)$. Let $\pi$ be a generic unramified representation of $\wt{\GL}_r$ such that $I(\mu) \onto \pi$ for some $\mu \in \Hom(Z(\wt{T}_r), \C)$. Let $\Theta(\wt{\GL}_R, \chi)$ be the distinguished theta representation associated to an exceptional character $\chi\in \Hom(Z(\wt{T}_R), \C)$. Let $\mca{W}^{\GL_R}_{\mca{O}_0}$ be the unique Whittaker model of $\Theta(\wt{\GL}_R, \chi)$. Then for any Whittaker model $\mca{W}^{\GL_r}$ of $\pi$, the following equality holds:
\begin{equation} \label{F:BH00}
\msc{Z} (s, \mca{W}^{\GL_r} \times \overline{\mca{W}_{\mca{O}_0}^{\GL_R}} )= L(n_\alpha s -\frac{R-1}{2}, \mu_{n_\alpha} \times \overline{\nu_R}) \cdot \mca{W}^{\GL_r}(1) \cdot \overline{ \mca{W}^{\GL_R}_{\mca{O}_0} }(1).
\end{equation}
Here $\mu_{n_\alpha}$ is the linear character associated to $\mu$ as in (\ref{F:lin}), and $\chi_{n_\alpha}[\nu_R]$ the linear character to $\chi$ in Lemma \ref{L:exp-ch}.
\end{conj}
\vskip 5pt

\begin{rmk}
A fundamental pair always satisfies $n|(2\p \cdot n_\alpha)$. The stronger assumption $n|(\p\cdot n_\alpha)$ is for technical reasons of the well-definedness of $\mu_{n_\alpha}$ as in (\ref{F:lin}). We believe that for general fundamental pair, an analogous Bump-Hoffstein conjecture still holds, though with the right hand side of (\ref{F:BH00}) presumably much involved. 

For a fundamental pair of Kazhdan-Patterson covering groups, the condition $n|(\p\cdot n_\alpha)$ is automatically satisfied since $n_\alpha=n$ in this case. The Bump-Hoffstein conjecture above is proven for theta representations in \cite{Suz1}. It is proven for arbitrary generic unramified representation $\pi$ of $\wt{\GL}_r$ by \cite{BF} and \cite{Gin}. Also, the $L$-function that appears in Conjecture \ref{C:GBH} is conjecturally equal to $L(n_\alpha s -\frac{R-1}{2}, \text{Sh}(\pi) \times \overline{\nu_R})$ where Sh($\pi$) is the sought Shimura lifting of $\pi$. 

For a fundamental pair of special Savin coverings satisfying $n|(\p\cdot n_\alpha)$, which is equivalent to the oddness of $n$ in this case, we have that $L(s, \mu_{n_\alpha} \times \overline{\nu_R})$ is nothing but $L(s, \pi \times \overline{\nu_R})$.
\end{rmk}

In the following section, we will show that Conjecture \ref{C:GBH} holds when $\pi$ is a generic theta representation of $\wt{\GL}_r$.
\vskip 10pt

\section{The case of theta representations}
We retain the notations in the previous section. Let $(\wt{\GL}_r^{(n)}, \wt{\GL}_R^{(n)})$ be a fundamental pair satisfying $n|(\p\cdot n_\alpha)$. In this section, $\pi=\Theta(\wt{\GL}_r, \mu)$ is a generic theta representation, where $\mu$ is an exceptional character for $\wt{\GL}_r$. We have
\begin{equation} \label{F:mu}
\mu_{n_\alpha}[\nu_r]=\bigotimes_{i=1}^r |\cdot |^{\frac{r+1}{2} + \nu_r -i}
\end{equation}
for some $\nu_r\in \C$. We also have the distinguished theta representation $\Theta(\wt{\GL}_R, \chi)$ with
\begin{equation} \label{F:chi}
\chi_{n_\alpha}[\nu_R]=\bigotimes_{i=1}^R |\cdot |^{\frac{R+1}{2} + \nu_R -i},
\end{equation}
where $\nu_R \in \C$.
\vskip 5pt

First, note that
\begin{align} \label{F:BH}
& \msc{Z}(s, \overline{\mca{W}_{\mca{O}_0}^{\GL_R}} \times \mca{W}_\cc^{\GL_r}) \\
=&\int_{U_r \backslash \GL_r} \overline{\mca{W}_{\mca{O}_0}^{\GL_R}}( \phi(\wt{g})) \cdot \mca{W}_\cc^{\GL_r}(\wt{g}) \cdot  |\det(g)|^{s-\frac{R-r}{2}} dg \nonumber \\
=&\int_{T_r} \overline{\mca{W}_{\mca{O}_0}^{\GL_R}}( \phi(\wt{t})) \cdot \mca{W}_\cc^{\GL_r}(\wt{t}) \cdot  |\det(t)|^s \cdot \delta^{1/2}_{B_r}(\wt{t})^{-1} \cdot \delta^{1/2}_{B_R}(\phi(\wt{t}))^{-1}  dt  \nonumber \\
=&\sum_{y\in Y_{\GL_r}} \overline{\mca{W}_{\mca{O}_0}^{\GL_R}}(\s_{\phi(y)}) \cdot \mca{W}_\cc^{\GL_r}(\s_y) \cdot  |\det(\s_y)|^s \cdot \delta^{1/2}_{B_r}(\s_y)^{-1} \cdot \delta^{1/2}_{B_R}(\s_{\phi(y)})^{-1}. \nonumber
\end{align}
As discussed before Proposition \ref{P:cc-val}, any Whittaker model $\mca{W}_\cc^{\GL_r}$ of $\Theta(\wt{\GL}_r, \mu)$  arises from an element
$$\cc=\sum_{i=1}^{ \val{\wp(\OF_{Q,n})} } a_i \cdot \cc^{\GL_r}_{\mca{O}_{y_i}} \in \Ftn(i(\mu)),$$
where $a_i \in \C$ and the $\cc_{\mca{O}_{y_i}}^{\GL_r}$'s are given in (\ref{F:cc}). Here $\set{\mca{O}_{y_i} } \subseteq \OF_{Q,n}$ is a set of representatives for $\wp(\OF_{Q,n})$. It follows that
$$\msc{Z}(s, \overline{\mca{W}_{\mca{O}_0}^{\GL_R}} \times \mca{W}_\cc^{\GL_r}) =
\sum_{i=1}^{ \val{\wp(\OF_{Q,n})} } a_i \cdot \msc{Z}(s, \overline{\mca{W}_{\mca{O}_0}^{\GL_R}} \times \mca{W}_{\mca{O}_{y_i}}^{\GL_r}).$$

\begin{lm} \label{L:O-0}
If $\wp_{\GL_r}(\mca{O}_{y_i}) \neq \wp_{\GL_r}(\mca{O}_0)$, then $\mca{W}^{\GL_r}_{\mca{O}_{y_i}}(1)=0$. Moreover in this case,
$$\msc{Z}(s, \overline{\mca{W}_{\mca{O}_0}^{\GL_R}} \times \mca{W}_{\mca{O}_{y_i}}^{\GL_r})=0.$$
\end{lm}
\begin{proof}
Assume $\wp_{\GL_r}(\mca{O}_{y_i}) \neq \wp_{\GL_r}(\mca{O}_0)$, it is easy to see that $\mca{W}_{\mca{O}_{y_i}}^{\GL_r}(1)=0$. For the second assertion, we will show that the supports of $\mca{W}_{\mca{O}_0}^{\GL_R}$ and $\mca{W}_{\mca{O}_{y_i}}^{\GL_r}$ are disjoint, and this will give the desired result. It suffices to consider $\wt{t}:=\s(\varpi^y) \in \wt{T}_r$ with $y\in Y_{\GL_r}$. Note that by Proposition \ref{P:Wh-Th}, $\mca{W}_{\mca{O}_{y_i}}^{\GL_r}(\wt{t})=0$ unless $\wt{t}$ is dominant and 
$$\w_{0,r}(y) \in \mca{O}_{y_i} + Y_{\GL_r, Q, n};$$ 
that is, there exists $\w\in W_{\GL_r}$ such that 
$$y \in \widehat{\w}(\widehat{y_i} + \rho_{\GL_r}) -\rho_{\GL_r} + Y_{\GL_r, Q,n}.$$
Now, suppose that $\phi(\wt{t})$ lies in the support of $\mca{W}_{\mca{O}_0}^{\GL_R}$, then $\w_{0,R}(\phi(y)) \in \mca{O}_0 + Y_{\GL_R, Q, n}$. That is, there exists some $\w'\in W_{\GL_R}$ such that
$$\phi(y) \in \widehat{\w'}(\rho_{\GL_R}) - \rho_{\GL_R} + Y_{\GL_R, Q,n},$$
where $Y_{\GL_R, Q, n}=n_\alpha \cdot Y_{\GL_R}$. Since $y\in Y_{\GL_r}$, it is easy to check that we have $\widehat{\w'} \in W_{\GL_r}$ and therefore
$$\phi(y) - (\widehat{\w'}(\rho_{\GL_R}) -\rho_{\GL_R})=\phi(y) - (\widehat{\w'}(\rho_{\GL_r}) -\rho_{\GL_r}) \in Y_{\GL_r} \cap Y_{\GL_R, Q,n}$$
with $Y_{\GL_r}\cap Y_{\GL_R,Q,n}=n_\alpha\cdot Y_{\GL_r} \subset Y_{\GL_r, Q,n}$ in this case. Hence, we have
$$\widehat{\w}(\widehat{y_i} + \rho_{\GL_r} ) -\rho_{\GL_r} \equiv \widehat{\w'}(\rho_{\GL_r}) -\rho_{\GL_r} \mod Y_{\GL_r,Q,n},$$
and this implies that
$$y_i \in \w''[0] + Y_{\GL_r,Q,n}$$
for some $\w'' \in W_{\GL_r}$, i.e., $\wp_{\GL_r}(\mca{O}_{y_i})=\wp_{\GL_r}(\mca{O}_0)$. This is a contradiction. Therefore we have shown that the supports of $\mca{W}_{\mca{O}_0}^{\GL_R}$ and $\mca{W}_{\mca{O}_{y_i}}^{\GL_r}$ are disjoint. The proof is completed.
\end{proof}

Immediately it follows from Lemma \ref{L:O-0} that for the Bump-Hoffstein conjecture it suffices to consider $\msc{Z}(s, \overline{\mca{W}_{\mca{O}_0}^{\GL_R}} \times \mca{W}_{\mca{O}_0}^{\GL_r})$.

For any $\w\in W_{\GL_R}$, define the set
\begin{equation} \label{dWh}
Y_{\GL_r, \w}:=
\left\{
\begin{array}{cc}
y=\sum_i y_i e_i \in Y_{\GL_r}: \\
\bullet \ \angb{\phi(y)}{\alpha}\ge 0 \text{ for all } \alpha\in \Delta_{\GL_R}, \\
\bullet \ \widehat{\phi(y)} \in \widehat{\w}[0] + Y_{\GL_R,Q,n}
\end{array}\right\}.
\end{equation}
The condition that $\angb{\phi(y)}{\alpha}\ge 0$ for all $\alpha\in \Delta_{\GL_R}$ is equivalent to 
$$y_1\ge y_2 \ge ... \ge y_r \ge 0.$$
Clearly, $\mca{W}_{\mca{O}_0}^{\GL_R}(\s_{\phi(y)})=0$ unless $y\in Y_{\GL_r,\w}$ for some $\w \in W_{\GL_R}$.
\vskip 10pt

The following lemma plays a crucial rule in the paper.
\begin{lm} \label{L:crucial}
With notations above, one has $Y_{\GL_r, \w}=\emptyset$ unless $\w\in W_{\GL_r}$. 
\end{lm}
\begin{proof}
The result relies on the crucial fact that $Y_{\GL_R, Q,n}=n_\alpha\cdot Y_{\GL_R}$ and $n_\alpha=R>r$ for a fundamental pair $(\wt{\GL}_r^{(n)}, \wt{\GL}_R^{(n)})$. We outline the argument.

First, for $y\in Y_{\GL_r}$ and $\w \in W_{\GL_R}$, the assertion that $\widehat{\phi(y)} \in \widehat{\w}[0] + Y_{\GL_R,Q,n}$ is equivalent to $y \in \w(\rho_{\GL_R}) -\rho_{\GL_R} + Y_{\GL_R, Q, n}$, which is further equivalent to
$$-y \in \w[0] + R\cdot Y_{\GL_R}.$$

For our purpose, define a function $f: Y_{\GL_R} \to \N_{\ge 0}$ by
$$f\left( \sum_{i=1}^R y_i e_i \right):=\text{max} \set{\val{y_i}: r < i \le R }.$$
For any $\w\in W_{\GL_R}$, write $\w[0]=\sum_{i=1}^R y_i e_i$; it is easy to see that $\text{max} \set{\val{y_i}: 1\le i\le R} \le R-1$ in this case. Therefore, $f(\w[0]) \in [0, R-1]$ for any $\w\in W_{\GL_R}$.

To show the Lemma, it suffices to show that if $f(\w[0])=0$, then $\w \in W_{\GL_r}$. However, if $\w \notin W_{\GL_r}$, then there must exist $y_i$ with $i>r$ of $\w[0]=\sum_i y_i e_i$ such that $y_i < 0$. That is, $f(\w[0]) \ge 1$. This completes the proof.
\end{proof}

It follows immediately that
\begin{align*}
& \msc{Z}(s, \overline{\mca{W}_{\mca{O}_0}^{\GL_R}} \times \mca{W}_{\mca{O}_0}^{\GL_r})\\
=&\sum_{\w \in W_{\GL_r}} \sum_{y\in Y_{\GL_r,\w}} \overline{\mca{W}_{\mca{O}_0}^{\GL_R}}(\s_{\phi(y)}) \cdot \mca{W}_{\mca{O}_0}^{\GL_r}(\s_y) \cdot |\det(\s_y)|^s \cdot \delta_{B_r}^{-1/2}(\s_y) \cdot \delta_{B_R}^{-1/2}(\s_{\phi(y)})
\end{align*}

To proceed, we define
$$Y_{\GL_r}^\dag=\set{y\in n_\alpha \cdot Y_{\GL_r}: \angb{\phi(y)}{\alpha}\ge 0 \text{ for all } \alpha \in \Delta_{\GL_R}}.$$
The proof of the following lemma is elementary (see also \cite[\S 4]{Suz1}), and we omit the details.
\begin{lm} 
For every $\w\in W_{\GL_r}$, there exists a unique $y_\w \in Y_{\GL_r,\w}$ such that 
$$Y_{\GL_r,\w}= y_\w + Y_{\GL_r}^\dag.$$
\end{lm}
By Lemma \ref{L:02}, we see that $y_\w=-\w[0] + z_\w$ for some $z_\w \in n_\alpha \cdot Y_{\GL_r}$. Here $z_\w$ is described as follows (see \cite[pg. 164]{Suz1}):
\begin{enumerate}
\item[$\bullet$] First, identify $\w \in W_{\GL_r}$ with a permutation of $\set{1, 2, ..., r}$ such that
$$\w \Big( \sum_i k_i e_i \Big)= \sum_i k_{\w^{-1}(i)} e_i. $$
Consider the following set
$$I_\w=\set{k\in [1, r-1] \cap \Z: \ \w^{-1}(k) > \w^{-1}(k+1)}.$$
\item[$\bullet$] Then, we have 
\begin{equation} \label{F:z-w}
z_\w=\sum_{k\in I_\w} n_\alpha\cdot (e_1+ e_2 + ... + e_k).
\end{equation}
\end{enumerate}
Again by Lemma \ref{L:02}, one has 
$$\widehat{y_\w}=\widehat{\w}[0] + \widehat{z_\w} \in Y_{\GL_r}$$
and
$$\widehat{\phi(y_\w)}=\widehat{\phi(\w)}[0] + \widehat{\phi(z_\w)} \in Y_{\GL_R}.$$

\begin{lm} \label{L:key2}
For any $\w\in W_{\GL_r}$, we have
$$\cc_{\mca{O}_0}^{\GL_r}(\s_{\widehat{y_\w}})=\prod_{k\in I_\w} q^{k\big(\frac{r+1}{2}-\nu_r \big)-\frac{k(k+1)}{2}} \cdot \cc^{\GL_r}_{\mca{O}_0}(\s_{\widehat{\w}[0]})$$
and
$$\cc_{\mca{O}_0}^{\GL_R}(\s_{\widehat{\phi(y_\w)}})=\prod_{k\in I_\w} q^{k\big(\frac{R+1}{2}-\nu_R \big)-\frac{k(k+1)}{2}} \cdot \cc^{\GL_R}_{\mca{O}_0}(\s_{\widehat{\phi(\w)}[0]}).$$
\end{lm}
\begin{proof}
We show the first equality, while the second is analogous. First, since $\widehat{y_\w}= \widehat{\w}[0] + \widehat{z_\w}$ with $\widehat{z_\w}\in n_\alpha \cdot Y_{\GL_r, Q,n}$, we have
$$\s_{\widehat{y_\w}}= \s_{\widehat{\w}[0]}\cdot  \s_{\widehat{z_\w}}.$$
and therefore
$$\cc_{\mca{O}_0}^{\GL_r}(\s_{\widehat{y_\w}})=\cc_{\mca{O}_0}^{\GL_r}(\s_{\widehat{\w}[0]}) \cdot \mu(\s_{\widehat{z_\w}}).$$
For any $y\in n_\alpha \cdot Y_{\GL_r}$, write $\mu(y):=\mu(\s_y)$ temporarily. By (\ref{F:mu}), we have
$$ \mu\left(\widehat{\sum_{i=1}^k n_\alpha e_i} \right) = \prod_{i=r+1-k}^r \mu(n_\alpha e_i) = \prod_{i=r+1-k}^r q^{-\big(\frac{r+1}{2} + \nu_r -i \big)} = \prod_{i=1}^k q^{\frac{r+1}{2} - \nu_r -i},$$
which is clearly equal to 
$$q^{k\big(\frac{r+1}{2}-\nu_r \big)-\frac{k(k+1)}{2}}.$$
In view of (\ref{F:z-w}), the proof is completed.
\end{proof}

To proceed, we define a function
$$\tau: W_{\GL_r} \to \C$$
by
$$\tau(\w):=l(\w) +  \sum_{k\in I_\w} k\cdot \left( n_\alpha s-\frac{R+r}{2} +\nu_r + \overline{\nu_R} +k \right).$$

\begin{cor} \label{C:key-mul}
For any $\w\in W_{\GL_r}$, one has
$$\overline{\cc^{\GL_R}_{\mca{O}_0}}(\s_{\widehat{\phi(y_\w)}}) \cdot \cc^{\GL_r}_{\mca{O}_0}(\s_{\widehat{y_\w}})  \cdot |\det(\s_{\widehat{y_\w}})|^s=q^{-\tau(\w)}.$$
\end{cor}
\begin{proof}
This follows from Corollary \ref{C:key1} and Lemma \ref{L:key2}.
\end{proof}

For any $2\le i_0\le r$, we put $\sigma_{i_0}:=(1 2... i_0) \in W_{\GL_r}$ for the $i_0$-cycle. Then 
 \begin{equation} \label{F:tau-cyc}
\tau(\sigma_{i_0})=n_\alpha s-\frac{R+r}{2} +\nu_r + \overline{\nu_R} + i_0.
\end{equation}

\begin{lm}[{ \cite[Lemma 4]{Suz1} }]  \label{L:suz}
With the above notation, we have
$$\tau(\sigma_2^{i_2} \sigma_3^{i_3} ... \sigma_j^{i_j} ... \sigma_r^{i_r})=i_2 \tau(\sigma_2) + ... + i_j \tau(\sigma_j) + ...+ i_r \tau(\sigma_r),$$
where $0\le i_j \le j-1$ for any $2\le j\le r$.
\end{lm}
Note that $W_{\GL_r}=\set{\sigma_2^{i_2} \sigma_3^{i_3} ... \sigma_j^{i_j} ... \sigma_r^{i_r}: \ 0\le i_j \le j-1 \text{ for all } 2\le j \le r}$. The zeta function $\zeta(s)$ is defined as $\zeta(s):=(1-q^{-s})^{-1}, s\in \C$. From this we obtain the following result.

\begin{prop} \label{P:sumW}
With notations above, one has
$$\sum_{\w \in W_{\GL_r}}  \overline{\cc^{\GL_R}_{\mca{O}_0}}(\s_{\widehat{\phi(y_\w)}}) \cdot \cc^{\GL_r}_{\mca{O}_0}(\s_{\widehat{y_\w}})  \cdot |\det(\s_{\widehat{y_\w}})|^s=\prod_{j=2}^r \frac{\zeta(\tau(\sigma_j))}{\zeta(j\cdot \tau(\sigma_j))}.$$
\end{prop}
\begin{proof}
By Corollary \ref{C:key-mul}, the left hand side is equal to
\begin{align*}
& \sum_{\w \in W_{\GL_r}} q^{-\tau(\w)} \\
=& \sum_{i_2, i_3, ..., i_r} q^{-\tau(\sigma_2^{i_2} \sigma_3^{i_3} ... \sigma_j^{i_j} ... \sigma_r^{i_r})} \\
=& \sum_{i_2, i_3, ..., i_r} q^{-(\sum_{2\le j\le r} i_j \tau(\sigma_j))} \text{ by Lemma \ref{L:suz}}\\
=& \prod_{j=2}^r \big(1+ q^{-\tau(\sigma_j)} + ... +  q^{-(j-1)\tau(\sigma_j)} \big),
\end{align*}
which is clearly equal to
$$\prod_{j=2}^r \frac{\zeta(\tau(\sigma_j))}{\zeta(j\cdot \tau(\sigma_j))}.$$
This completes the proof.
\end{proof}

It follows from Proposition \ref{P:Wh-Th} that
\begin{align*}
& \msc{Z}(s, \overline{\mca{W}_{\mca{O}_0}^{\GL_R}} \times \mca{W}_{\mca{O}_0}^{\GL_r})\\
=&\sum_{\w \in W_{\GL_r}} \sum_{y\in Y_{\GL_r,\w}} \overline{\mca{W}_{\mca{O}_0}^{\GL_R}}(\s_{\phi(y)}) \cdot \mca{W}_{\mca{O}_0}^{\GL_r}(\s_y) \cdot |\det(\s_y)|^s \cdot \delta_{B_r}^{-1/2}(\s_y) \cdot \delta_{B_R}^{-1/2}(\s_{\phi(y)}) \\
=&\overline{\mca{W}_{\mca{O}_0}^{\GL_R}}(1) \cdot \mca{W}_{\mca{O}_0}^{\GL_r}(1) \cdot \sum_{\w \in W_{\GL_r}} \sum_{y\in Y_{\GL_r,\w}} \overline{\cc_{\mca{O}_0}^{\GL_R} }(\s_{\widehat{\phi(y)}}) \cdot \cc_{\mca{O}_0}^{\GL_r}(\s_{\widehat{y}}) \cdot |\det(\s_y)|^s \\
=&\overline{\mca{W}_{\mca{O}_0}^{\GL_R}}(1) \cdot \mca{W}_{\mca{O}_0}^{\GL_r}(1) \cdot \sum_{\w \in W_{\GL_r}} \sum_{y^\dag \in Y_{\GL_r}^\dag} \overline{\cc^{\GL_R}_{\mca{O}_0}}(\s_{\widehat{\phi(y_\w)}+\widehat{\phi(y^\dag)}}) \cdot \cc^{\GL_r}_{\mca{O}_0}(\s_{\widehat{y_\w}+ \widehat{y^\dag}}) \cdot |\det(\s_{\widehat{y_\w}})|^s \cdot |\det(\s_{\widehat{y^\dag}})|^s \\
=& \overline{\mca{W}_{\mca{O}_0}^{\GL_R}}(1) \cdot \mca{W}_{\mca{O}_0}^{\GL_r}(1) \\
& \quad  \cdot \sum_{y^\dag \in Y_{\GL_r}^\dag} \sum_{\w \in W_{\GL_r}}  \overline{\cc^{\GL_R}_{\mca{O}_0}}(\s_{\widehat{\phi(y_\w)}}) \cdot \overline{\chi}(\s_{\widehat{\phi(y^\dag)}}) \cdot \cc_{\mca{O}_0}^{\GL_r}(\s_{\widehat{y_\w}}) \cdot \mu(\s_{\widehat{y^\dag}}) \cdot |\det(\s_{\widehat{y_\w}})|^s \cdot |\det(\s_{\widehat{y^\dag}})|^s  \\
=& \overline{\mca{W}_{\mca{O}_0}^{\GL_R}}(1) \cdot \mca{W}_{\mca{O}_0}^{\GL_r}(1) \cdot \Big( \sum_{y^\dag \in Y_{\GL_r}^\dag} \overline{\chi}(\s_{\widehat{\phi(y^\dag)}}) \cdot \mu(\s_{\widehat{y^\dag}}) \cdot |\det(\s_{\widehat{y^\dag}})|^s \Big) \\
& \quad \cdot \Big( \sum_{\w \in W_{\GL_r}}  \overline{ \cc^{\GL_R}_{\mca{O}_0}}(\s_{\widehat{\phi(y_\w)}}) \cdot \cc^{\GL_r}_{\mca{O}_0}(\s_{\widehat{y_\w}})  \cdot |\det(\s_{\widehat{y_\w}})|^s \Big).
\end{align*}
To complete the computation for $\msc{Z}(s, \overline{\mca{W}_{\mca{O}_0}^{\GL_R}} \times \mca{W}_{\mca{O}_0}^{\GL_r})$, we show the following result.

\begin{prop} \label{P:sum-dag}
Let $\chi$ (resp. $\mu$) be an unramified exceptional character of $\wt{\GL}_{R}$ (resp. $\wt{\GL}_r$). Then
$$\sum_{y^\dag \in Y_{\GL_r}^\dag} \overline{\chi}(\s_{\widehat{\phi(y^\dag)}}) \cdot \mu(\s_{\widehat{y^\dag}}) \cdot |\det(\s_{\widehat{y^\dag}})|^s=\prod_{j=1}^r \zeta(j\cdot \tau(\sigma_j)).$$
\end{prop}
\begin{proof}
Any element of $Y_{\GL_r}^\dag$ is of the form $y^\dag=\sum_{i=1}^r (n_\alpha k_i) e_i$ with 
$$k_i=l_i+l_{i+1}+ ... + l_r$$
such that $l_i\ge 0$ for all $i$. We have $\widehat{\phi(y^\dag)}=\sum_{i=1}^r (n_\alpha \cdot k_{n+1-i})e_i$ and $\widehat{y^\dag}= \sum_{i=1}^r (n_\alpha \cdot k_{r+1-i}) e_i$. It follows that
\begin{align*}
& \sum_{y^\dag \in Y_{\GL_r}^\dag} \overline{\chi}(\s_{\widehat{\phi(y^\dag)}})\cdot \mu(\s_{\widehat{y^\dag}}) \cdot |\det(\s_{\widehat{y^\dag}})|^s \\
=&  \sum_{ k_i \text{ as above}} \prod_{i=1}^r \overline{\chi}(\s_{n_\alpha e_{R+1-i}})^{k_i} \cdot \mu(\s_{n_\alpha e_{r+1-i}})^{k_i} \cdot q^{-n_\alpha (\sum_{i=1}^r k_i)s} \\
=& \sum_{ k_i \text{ as above}}q^{-n_\alpha(\sum_{i=1}^r k_i)s}  \cdot  \prod_{i=1}^r  \left(  q^{\frac{R+1}{2}-i - \overline{\nu_R}} \cdot q^{\frac{r+1}{2}-i - \nu_r } \right)^{k_i}\\
=& \sum_{ k_i \text{ as above}}q^{-n_\alpha (\sum_{i=1}^r k_i)s}  \cdot  \prod_{i=1}^r  \left( q^{\frac{R+r+2}{2}-2i - \overline{\nu_R} - \nu_r } \right)^{k_i} \\
=& \sum_{ k_i \text{ as above}} \prod_{i=1}^r  \left( q^{-n_\alpha s + \frac{R+r+2}{2}-2i - \overline{\nu_R} - \nu_r } \right)^{k_i}.
\end{align*}
By substituting $k_i=l_i+ l_{i+1} + ... + l_r$, the above sum is equal to
\begin{align*}
& \sum_{ l_j\ge 0}  \prod_{i=1}^r \prod_{j\ge i}  \left( q^{-n_\alpha s + \frac{R+r+2}{2}-2i - \overline{\nu_R} - \nu_r } \right)^{l_j} \\
=& \sum_{ l_j\ge 0}   \prod_{j=1}^r \prod_{1\le i \le j} \left( q^{-n_\alpha s + \frac{R+r+2}{2}-2i - \overline{\nu_R} - \nu_r } \right)^{l_j}  \\
=&\sum_{ l_j\ge 0}   \prod_{j=1}^r \left( q^{j(-n_\alpha s + \frac{R+r}{2}-j - \overline{\nu_R} - \nu_r) } \right)^{l_j} \\
=&\sum_{ l_j\ge 0}   \prod_{j=1}^r \left(q^{-j\cdot \tau(\sigma_j)} \right)^{l_j} \\
=& \prod_{j=1}^r \zeta \left( j\cdot \tau(\sigma_j)  \right),
\end{align*}
where the second last equality follows from the fact that $\tau(\sigma_j)= n_\alpha s-\frac{R+r}{2}+ j+ \nu_r +\overline{\nu_R}$ as in (\ref{F:tau-cyc}). This completes the proof.
\end{proof}

We give a summary of the main result of this section.
\begin{thm}  \label{T:theta}
Let $(\wt{\GL}_r, \wt{\GL}_R)$ be a fundamental pair of $n$-fold coverings satisfying $n|(\p\cdot n_\alpha)$. Let $\Theta(\wt{\GL}_r, \mu)$ and $\Theta(\wt{\GL}_R, \chi)$ be theta representations of $\wt{\GL}_r$ and $\wt{\GL}_R$ respectively. Let $\mu_{n_\alpha}[\nu_r]$ and $\chi_{n_\alpha}[\nu_R]$ be the linear characters associated to $\mu$ and $\chi$ as in (\ref{F:mu}) and (\ref{F:chi}) respectively. Let $\mca{W}_{\mca{O}_0}^{\GL_r}$ and $\mca{W}^{\GL_R}_{\mca{O}_0}$ be the Whittaker models for the two theta representations. Then,
$$ \msc{Z}(s, \overline{\mca{W}_{\mca{O}_0}^{\GL_R}} \times \mca{W}_{\mca{O}_0}^{\GL_r})=L\big( n_\alpha s-\frac{R-1}{2}, \mu_{n_\alpha} \times \overline{\nu_R} \big) \cdot \overline{\mca{W}_{\mca{O}_0}^{\GL_R}}(1) \cdot \mca{W}_{\mca{O}_0}^{\GL_r}(1).$$
In view of Lemma \ref{L:O-0}, we see that the generalized Bump-Hoffstein conjecture \ref{C:GBH} holds for theta representations.
\end{thm}
\begin{proof}
By Proposition \ref{P:sumW} and \ref{P:sum-dag}, we have
\begin{align*}
& \msc{Z}(s, \overline{\mca{W}_{\mca{O}_0}^{\GL_R}} \times \mca{W}_{\mca{O}_0}^{\GL_r}) \\
=& \overline{\mca{W}_{\mca{O}_0}^{\GL_R}}(1) \cdot \mca{W}_{\mca{O}_0}^{\GL_r}(1) \cdot \prod_{j=1}^r \zeta\left(\tau(\sigma_j)   \right) \\
=& \overline{\mca{W}_{\mca{O}_0}^{\GL_R}}(1) \cdot \mca{W}_{\mca{O}_0}^{\GL_r}(1) \cdot \prod_{j=1}^r \zeta \left( n_\alpha s-\frac{R+r}{2}+ j+ \nu_r +\overline{\nu_R}  \right) \\
=& \overline{\mca{W}_{\mca{O}_0}^{\GL_R}}(1) \cdot \mca{W}_{\mca{O}_0}^{\GL_r}(1) \cdot \prod_{i=1}^r \zeta \left( n_\alpha s-\frac{R-r}{2}+ 1-i+ \nu_r +\overline{\nu_R}  \right) \\
=& \overline{\mca{W}_{\mca{O}_0}^{\GL_R}}(1) \cdot \mca{W}_{\mca{O}_0}^{\GL_r}(1) \cdot \prod_{i=1}^r \zeta \left( n_\alpha s-\frac{R-1}{2}+ \left( \frac{r+1}{2} -i + \nu_r \right) +\overline{\nu_R}  \right) \\
=& \overline{\mca{W}_{\mca{O}_0}^{\GL_R}}(1) \cdot \mca{W}_{\mca{O}_0}^{\GL_r}(1) \cdot L(n_\alpha s -\frac{R-1}{2}, \mu_{n_\alpha}[\nu_r] \times \overline{\nu_R}).
\end{align*}
The proof is completed.
\end{proof}

\section{The rank two case}
In this section, let $(\wt{\GL}_2, \wt{\GL}_R)$ be a fundamental pair of $n$-fold covers satisfying $n|(\p \cdot n_\alpha)$, i.e., $r=2$. Let $\pi$ be any generic unramified representation of the $\wt{\GL}_2$, and $\Theta(\wt{\GL}_R, \chi)$ the theta representation associated to an exceptional character $\chi$. Let $\chi_{n_\alpha}[\nu]$ be the linear character associated to $\chi$, where $\nu\in \C$. We show in this section that the Bump-Hoffstein conjecture holds for $\pi$. To avoid confusion, we may continue to use the notation $\wt{T}_r, \wt{A}_r$ etc, with $r=2$ understood.

Note that $I(\mu) \onto \pi$ where $\mu$ is an unramified character  of $Z(\wt{T}_r)$. A Whittaker model for $\pi$ is just $\mca{W}^{\GL_2}_\cc$ for some $\cc \in \Ftn(i(\mu))$. Let $\mu_{n_\alpha}$ be the linear character in (\ref{F:lin}). Therefore, to show the Bump-Hoffstein conjecture in this case, it suffices to consider
$$\msc{Z}(s, \overline{\mca{W}^{\GL_R}_{\mca{O}_0}} \times \mca{W}^{\GL_2}_\gamma),$$
where $\mca{W}^{\GL_2}_\gamma$ is the Whittaker model of $I(\mu)$ associated with $\gamma \in \wt{T}_r$.
\vskip 10pt

Let $W_{\GL_2}=\set{\text{id}, \w_1}$ be the Weyl group of $\GL_2$ with $\w_1:=\w_{\alpha_1}$. In this section, we write
$$\mu_{\alpha}:=\mu(\wt{h}_{\alpha_1}(\varpi^{n_\alpha})).$$
By Proposition \ref{P:Pat}, for any $\gamma \in \wt{T}_r$ and dominant $\wt{t}$,
\begin{align} \label{F:Wh2}
& \mca{W}_\gamma^{\GL_2}( \wt{t} ) \cdot \delta_{B_2}^{-1/2}( \wt{t} )\\
 = & \gk(\w_1, \mu) \cdot \tau(\text{id}, \mu, \gamma, w_1 \wt{t} w_1^{-1}) + \gk(\text{id}, \mu) \cdot \tau(\w_1, {}^{\w_1}\mu, \gamma, w_1 \wt{t} w_1^{-1}) \nonumber \\
=& \frac{1-q^{-1}\mu_{\alpha}}{1-\mu_{\alpha}} \cdot  \tau(\text{id}, \mu, \gamma, w_1 \wt{t} w_1^{-1}) + \tau(\w_1, {}^{\w_1}\mu, \gamma, w_1 \wt{t} w_1^{-1}) \nonumber .
\end{align}

We have a reduction arising from the following result.
\begin{lm}
If $\gamma \notin (\s_0 \cdot \wt{A}_r)\cup (\s_{\w_1[0]} \cdot \wt{A}_r)$, then $\mca{W}_\gamma^{\GL_2}(1)=0$ and 
\begin{equation} \label{F:2-red}
\overline{\mca{W}^{\GL_R}_{\mca{O}_0}}(\phi(\wt{t})) \cdot \mca{W}^{\GL_2}_\gamma(\wt{t})=0
\end{equation}
for any $\wt{t}\in \wt{T}_r$.
\end{lm}
\begin{proof}
One has that
$$\mca{W}^{\GL_2}_\gamma(1)=\frac{1-q^{-1}\mu_{\alpha}}{1-\mu_{\alpha}} \cdot  \tau(\text{id}, \mu, \gamma, \s_0) + \tau(\w_1, {}^{\w_1}\mu, \gamma, \s_0).$$
If $\gamma \notin (\s_0 \cdot \wt{A}_r)\cup (\s_{\w_1[0]} \cdot \wt{A}_r)$, then it follows from Theorem \ref{T:SLCM} that
$$\tau(\text{id}, \mu, \gamma, \s_0)=0 \text{ and } \tau(\w_1, {}^{\w_1}\mu, \gamma, \s_0)=0.$$
Therefore  $\mca{W}^{\GL_2}_\gamma(1)=0$ in this case. The proof for (\ref{F:2-red}) is similar to that in Lemma \ref{L:O-0} and we omit the details.
\end{proof}

It follows immediately that it suffices to consider the case where $\gamma \in \s_0 \cdot \wt{A}_r$ or $\gamma \in \s_{\w_1[0]} \cdot \wt{A}_r$. However, for any $\wt{a}\in \wt{A}$, we have
$$\mca{W}^{\GL_2}_{\gamma \cdot \wt{a}} =\mu(\wt{a})^{-1} \cdot \mca{W}^{\GL_2}_\gamma.$$
Therefore it suffices to consider the cases where
$$\gamma=\s_{0} \text{ or } \s_{\w_1[0]}.$$
For $\gamma=\s_{0} \text{ or } \s_{\w_1[0]}$, by (\ref{F:Wh2}), we have
$$\overline{ \mca{W}_{\mca{O}_0}^{\GL_R} }(\s_{\phi(y)}) \cdot \mca{W}_\gamma^{\GL_2}(\s_y)=0$$
unless that $y$ belongs to
$$Y_{\GL_2, \w}=\set{ y=y_1e_1 + y_2e_2:  y_1\ge y_2\ge 0, \quad \widehat{\phi(y)} \in \widehat{\phi(\w)}[0] + Y_{\GL_R, Q,n} }.$$
More explicitly,
$$Y_{\GL_2,\text{id}}=\set{ n_\alpha(k_1e_1 + k_2 e_2): k_1 \ge k_2 \ge 0  }$$
and
$$Y_{\GL_2, \w_1}=\set{ (n_\alpha k_1-1)e_1 + (n_\alpha k_2 +1)e_2: k_1 \ge k_2 +1, k_2 \ge 0  }.$$

\vskip 5pt
Now we discuss the two cases $\gamma = \s_0=1$ or $\gamma=\s_{\w_1[0]}$ separately, and the result follows from explicit computations.
\subsection{Case $\gamma=\s_0$} 
First, we consider an element $y=n_\alpha k_1 e_1 + n_\alpha k_2 e_2$ in $Y_{\GL_2, \text{id}}$. From (\ref{F:Wh2}), we have
\begin{align*}
& \mca{W}^{\GL_2}_{\s_0}(\s_{y}) \cdot \delta_{B_2}^{-1/2}(\s_{y})\\
=& \frac{1-q^{-1}\mu_{\alpha}}{1-\mu_{\alpha}} \cdot  \tau(\text{id}, \mu, \s_0, w_1 \cdot \s_y \cdot w_1^{-1}) + \tau(\w_1, {}^{\w_1}\mu, \s_0, w_1\cdot \s_y \cdot w_1^{-1})  \\
=& \frac{1-q^{-1}\mu_{\alpha}}{1-\mu_{\alpha}} \cdot  \tau(\text{id}, \mu, \s_0 , \s_{\w_1(y)}) + \tau(\w_1, {}^{\w_1}\mu, \s_0, \s_{\w_1(y)})  \text{ by Lemma \ref{L:2n}} \\
=& \frac{1-q^{-1}\mu_\alpha}{1-\mu_\alpha} \cdot \mu(\s_{n_\alpha e_1})^{k_2} \cdot \mu(\s_{n_\alpha e_2})^{k_1} + \frac{1-q^{-1}}{1-\mu_\alpha^{-1}} \cdot \mu(\s_{n_\alpha e_1})^{k_1} \cdot \mu(\s_{n_\alpha e_2})^{k_2},
\end{align*}
where the last equality follows from Theorem \ref{T:SLCM} and the fact that $y\in n_\alpha \cdot Y_{\GL_2} \subseteq Y_{\GL_2,Q,n}$. In particular, 
$$\mca{W}_{\s_0}^{\GL_2}(1)=1.$$

On the other hand,
$$\tau(\text{id}, \mu, \s_0, \s_{\w_1[0]})=0$$
and 
$$\tau(\w_1, {}^{\w_1}\mu, \s_0, \s_{\w_1[0]})=\g(\angb{\w_1[0]_\rho}{\alpha_1} Q(\alpha^\vee)) =\g(Q(\alpha^\vee)).$$
It follows that for any element $z:=(n_\alpha k_1-1)e_1 + (n_\alpha k_2 +1)e_2=-\w_1[0] + n_\alpha k_1 e_1 + n_\alpha k_2 e_2$ in $Y_{\GL_2,\w_1}$, one has
\begin{align*}
& \mca{W}_{\s_0}^{\GL_2}(\s_{z})\cdot \delta_{B_2}^{-1/2}(\s_{z})\\
= & \tau(\w_1, {}^{\w_1}\mu, \s_0, \s_{\widehat{-\w_1[0]}} )\cdot {}^{\w_1}\mu( \s_{n_\alpha k_2e_1 + n_\alpha k_1 e_2} ) \\
=& \tau(\w_1, {}^{\w_1}\mu, \s_0, \s_{\w_1[0]} )\cdot \mu( \s_{n_\alpha k_1e_1} ) \cdot \mu( \s_{n_\alpha k_2 e_2} )  \\
=& \g(Q(\alpha^\vee)) \cdot \mu(\s_{n_\alpha e_1})^{k_1} \cdot \mu(\s_{n_\alpha e_2})^{k_2}.
\end{align*}
As in (\ref{F:BH}), we first compute
\begin{align} \label{F:01}
& \sum_{y\in Y_{\GL_2,\text{id}}} \mca{W}_{\s_0}^{\GL_2}(\s_y) \cdot \overline{\mca{W}_{\mca{O}_0}^{\GL_R} } (\s_{\phi(y)}) \cdot |\det(\s_y)|^s \cdot \delta_{B_2}^{-1/2}(\s_y) \cdot \delta_{B_R}^{-1/2}(\s_{\phi(y)}) \\
=&  \sum_{ \substack{y=n_\alpha k_1 e_1 +n_\alpha k_2 e_2 \\ k_1\ge k_2 \ge 0 } } \mca{W}_{\s_0}^{\GL_2}(\s_y)\cdot \overline{\mca{W}_{\mca{O}_0}^{\GL_R} } (\s_{\phi(y)}) \cdot |\det(\s_y)|^s \cdot \delta_{B_2}^{-1/2}(\s_y) \cdot \delta_{B_R}^{-1/2}(\s_{\phi(y)}) \nonumber \\
=& \mca{W}_{\s_0}^{\GL_2}(1) \cdot \overline{\mca{W}_{\mca{O}_0}^{\GL_R} }(1)  \cdot  \overline{\chi}(\s_{n_\alpha e_{R-1}})^{k_2} \cdot  \overline{\chi}(\s_{n_\alpha e_R})^{k_1} \cdot q^{-n_\alpha s k_1- n_\alpha s k_2} \nonumber \\
& \quad \cdot \sum_{k_1\ge k_2 \ge 0 } \left( \frac{1-q^{-1}\mu_\alpha}{1-\mu_\alpha} \cdot \mu(\s_{n_\alpha e_1})^{k_2} \cdot \mu(\s_{n_\alpha e_2})^{k_1} + \frac{1-q^{-1}}{1-\mu_\alpha^{-1}} \cdot \mu(\s_{n_\alpha e_1})^{k_1} \cdot \mu(\s_{n_\alpha e_2})^{k_2} \right),  \nonumber 
\end{align}
where
$$\chi(\s_{n_\alpha e_{R-1}})=q^{-(\frac{R+1}{2}-(R-1) +\nu)} \text{ and } \chi(\s_{n_\alpha e_{R}})=q^{-(\frac{R+1}{2}-R +\nu)}.$$
\vskip 5pt

Write $\mu_{e_i}:=\mu_{\s_{n_\alpha e_i}}$ and similarly $\chi_{e_i}:=\chi_{\s_{n_\alpha e_i}}$. We see that (\ref{F:01}) is equal to  
$$\mca{W}_{\s_0}^{\GL_2}(1) \cdot \overline{\mca{W}_{\mca{O}_0}^{\GL_R} }(1)\cdot \left( \frac{1-q^{-1}\mu_\alpha}{1-\mu_\alpha} \cdot S_2 -  \frac{(1-q^{-1})\mu_\alpha}{1-\mu_\alpha} \cdot S_1\right) $$
where
\begin{align*}
S_2 &= \sum_{k_2\ge 0} \left( \mu_{e_1} q^{-n_\alpha s}\cdot q^{-(\frac{R+1}{2}-(R-1) +\overline{\nu})} \right)^{k_2} \sum_{k_1 \ge k_2}  \left( \mu_{e_2} q^{-n_\alpha s}\cdot q^{-(\frac{R+1}{2}-R + \overline{\nu})} \right)^{k_1}  \\
&= \frac{1}{(1- \mu_{e_2} q^{-(n_\alpha s + \frac{R+1}{2}-R +\overline{\nu})} ) \cdot (1- \mu_{e_1} \mu_{e_2} q^{-(2n_\alpha s + 2-R + 2\overline{\nu})}  )}
\end{align*}
and
\begin{align*}
S_1 &= \sum_{k_2\ge 0} \left( \mu_{e_2} q^{-n_\alpha s}\cdot q^{-(\frac{R+1}{2}-(R-1) +\overline{\nu})} \right)^{k_2} \sum_{k_1 \ge k_2}  \left( \mu_{e_1} q^{-n_\alpha s}\cdot q^{-(\frac{R+1}{2}- R + \overline{\nu})} \right)^{k_1}\\
&= \frac{1}{(1- \mu_{e_1} q^{-(n_\alpha s + \frac{R+1}{2}- R +\overline{\nu})} ) \cdot (1- \mu_{e_1} \mu_{e_2} q^{-(2n_\alpha s + 2-R + 2\overline{\nu})}  )} .
\end{align*}

Now we consider the sum over $Y_{\GL_2, \w_1}$ and obtain
\begin{align} \label{F:02}
& \sum_{y\in Y_{\GL_2, \w_1}} \mca{W}_{\s_0}^{\GL_2}(\s_y) \cdot \overline{\mca{W}_{\mca{O}_0}^{\GL_R} } (\s_{\phi(y)}) \cdot |\det(\s_y)|^s \cdot \delta_{B_2}^{-1/2}(\s_y) \cdot \delta_{B_R}^{-1/2}(\s_{\phi(y)})\\
=& \overline{\mca{W}_{\mca{O}_0}^{\GL_R} }(1) \cdot \sum_{ \substack{k_2\ge 0 \\ k_1\ge k_2 +1 } } \left(\g(Q(\alpha^\vee)) \mu_{e_1}^{k_1} \mu_{e_2}^{k_2} \right)  \cdot \left(q^{-1} \overline{\g(-Q(\alpha^\vee))^{-1}} \right) \cdot \overline{\chi}_{e_{R-1}}^{k_2} \overline{\chi}_{e_R}^{k_1} \cdot q^{-n_\alpha s k_1 - n_\alpha s k_2} \nonumber \\
=&\overline{\mca{W}_{\mca{O}_0}^{\GL_R} }(1) \cdot  q^{-1} \cdot   \sum_{k_2\ge 0} \left( \mu_{e_2} q^{-n_\alpha s}\cdot q^{-(\frac{R+1}{2}-(r-1) +\overline{\nu})} \right)^{k_2} \sum_{k_1 \ge k_2 +1}  \left( \mu_{e_1} q^{-n_\alpha s}\cdot q^{-(\frac{R+1}{2}- R + \overline{\nu})} \right)^{k_1}  \nonumber \\
=& \overline{\mca{W}_{\mca{O}_0}^{\GL_R} }(1) \cdot \mu_{e_1} \cdot q^{-(n_\alpha s + \frac{R+1}{2} + 1- R + \overline{\nu})} \cdot S_1 \nonumber  \\
=& \mu_{e_1} \cdot q^{-(n_\alpha s + \frac{R+1}{2} + 1- R + \overline{\nu})} \cdot S_1 \cdot  \mca{W}_{\s_0}^{\GL_2}(1) \cdot  \overline{\mca{W}_{\mca{O}_0}^{\GL_R} }(1). \nonumber
\end{align}
Combining (\ref{F:01}) and (\ref{F:02}) above, we get
\begin{align*}
& \msc{Z}(s, \mca{W}_{\s_0}^{\GL_2} \times \overline{\mca{W}_{\mca{O}_0}^{\GL_R} } ) \\
=& \sum_{y\in Y_{\GL_2, \text{id}} \cup Y_{\GL_2, \w_1}} \mca{W}_{\s_0}^{\GL_2}(\s_y) \cdot \overline{\mca{W}_{\mca{O}_0}^{\GL_R} } (\s_{\phi(y)}) \cdot |\det(\s_y)|^s \cdot \delta_{B_2}^{-1/2}(\s_y) \cdot \delta_{B_R}^{-1/2}(\s_{\phi(y)}) \\
=& \frac{1}{(1- \mu_{e_1} q^{-(n_\alpha s + \frac{R+1}{2}-R +\overline{\nu})} )} \cdot \frac{1}{(1- \mu_{e_2} q^{-(n_\alpha s + \frac{R+1}{2}-R +\overline{\nu})} )} \cdot  \mca{W}_{\s_0}^{\GL_2}(1) \cdot  \overline{\mca{W}_{\mca{O}_0}^{\GL_R} }(1)\\
=& L\left(n_\alpha s -\frac{R-1}{2}, \mu_{n_\alpha} \times \overline{\nu} \right) \cdot \mca{W}_{\s_0}^{\GL_2}(1) \cdot  \overline{\mca{W}_{\mca{O}_0}^{\GL_R} }(1).
\end{align*}

We summarize as follows.
\begin{prop} With notations above, we have
$$\msc{Z}(s, \mca{W}_{\s_0}^{\GL_2} \times \overline{\mca{W}_{\mca{O}_0}^{\GL_R} } )=L(n_\alpha s -\frac{R-1}{2}, \mu_{n_\alpha} \times \overline{\nu}) \cdot \mca{W}_{\s_0}^{\GL_2}(1) \cdot  \overline{\mca{W}_{\mca{O}_0}^{\GL_R} }(1).$$
\end{prop}

\vskip 5pt
\subsection{Case $\gamma=\s_{\w_1[0]}$}
We keep the notation $\mu_{e_i}:=\mu(\s_{n_\alpha e_i})$ for $i=1, 2$. In this case, for $y=n_\alpha k_1 e_1 + n_\alpha k_2 e_2$, it follows from (\ref{F:Wh2}) that
\begin{align*}
& \mca{W}^{\GL_2}_{\s_{\w_1[0]}}(\s_{})  \cdot \delta_{B_2}^{-1/2}(\s_{y})\\
= & \frac{1-q^{-1}\mu_\alpha}{1-\mu_\alpha} \cdot \tau(\text{id}, \mu, \s_{\w_1[0]}, \s_0) \cdot 
(\mu_{e_1})^{k_2} \cdot (\mu_{e_2})^{k_1} + \tau(\w_1, {}^{\w_1} \mu, \s_{\w_1[0]}, \s_0) (\mu_{e_1})^{k_1} \cdot (\mu_{e_2})^{k_2} \\
= & \g(-Q(\alpha^\vee)) \cdot (\mu_{e_1})^{k_1} \cdot (\mu_{e_2})^{k_2}.
\end{align*}
In particular, 
$$\mca{W}^{\GL_2}_{\s_{\w_1[0]}}(1)=\g(-Q(\alpha^\vee)).$$

On the other hand, consider $z=-\w_1[0]+n_\alpha k_1e_1 + n_\alpha k_2 e_2$. It follows from the two equalities
$\tau(\text{id}, \mu, \s_{\w_1[0]}, \s_{\w_1[0]})=1$
and 
$$\tau(\w_1, {}^{\w_1}\mu, \s_{\w_1[0]}, \s_{\w_1[0]})=\frac{(1-q^{-1})\cdot ({}^{\w_1}\mu)_\alpha}{1-({}^{\w_1}\mu)_\alpha}$$
that
\begin{align*}
& \mca{W}_{\s_{\w_1[0]}}^{\GL_2}(\s_{z}) \cdot \delta_{B_2}^{-1/2}(\s_{z})\\
= & \frac{1-q^{-1}\mu_\alpha}{1-\mu_\alpha} \cdot \mu_{e_1}^{k_2} \cdot \mu_{e_2}^{k_1} -  \frac{1-q^{-1}}{1-\mu_\alpha} \cdot \mu_{e_1}^{k_1} \cdot \mu_{e_2}^{k_2}.
\end{align*}
In particular, 
$$\mca{W}_{\s_{\w_1[0]}}^{\GL_2}(\s_{-\w_1[0]})=q^{-1}.$$
Therefore we have
\begin{align*}
& \sum_{y\in Y_{\GL_2, \w_1}} \mca{W}_{\s_{\w_1[0]}}^{\GL_2}(\s_y) \cdot \overline{\mca{W}_{\mca{O}_0}^{\GL_R} }(\s_{\phi(y)}) \cdot |\det(\s_y)|^s  \cdot \delta_{B_2}^{-1/2}(\s_y) \cdot \delta_{B_R}^{-1/2}(\s_{\phi(y)})\\
=& \overline{\mca{W}_{\mca{O}_0}^{\GL_R} }(1) \cdot  \sum_{ \substack{ k_2\ge 0 \\ k_1\ge k_2 +1} } \mca{W}_{\s_{\w_1[0]}}^{\GL_2}(\s_{-\w_1[0]+n_\alpha k_1e_1 + n_\alpha k_2 e_2}) \cdot \delta_{B_2}^{-1/2}(\s_{-\w_1[0]+n_\alpha k_1e_1 + n_\alpha k_2 e_2})  \\
& \quad \cdot \overline{\cc_{\mca{O}_0}^{\GL_R} } (\s_{\widehat{\phi(\w_1)}[0]+n_\alpha k_1e_R + n_\alpha k_2 e_{R-2}}) \cdot q^{-n_\alpha s k_1 - n_\alpha s k_2} \\
=& \overline{\mca{W}_{\mca{O}_0}^{\GL_R} }(1) \cdot \overline{\cc_{\mca{O}_0}^{\GL_R} } (\s_{\widehat{\phi(\w_1)}[0]})   \cdot \sum_{\substack{k_2\ge 0 \\ k_1\ge k_2 +1} } \left( \frac{1-q^{-1}\mu_\alpha}{1-\mu_\alpha} \cdot \mu_{e_1}^{k_2} \cdot \mu_{e_2}^{k_1} - \frac{1-q^{-1}}{1-\mu_\alpha} \cdot \mu_{e_1}^{k_1} \cdot \mu_{e_2}^{k_2} \right) \\
& \quad \cdot   \overline{\chi}_{e_{R-1}}^{k_2} \cdot  \overline{\chi}_{e_R}^{k_1} \cdot q^{-n_\alpha s k_1- n_\alpha s k_2} \\
=&\overline{\mca{W}_{\mca{O}_0}^{\GL_R} }(1) \cdot q^{-1} \cdot \overline{\g(-Q(\alpha^\vee))}^{-1} \cdot  \sum_{\substack{k_2\ge 0 \\ k_1\ge k_2 +1} } \left( \frac{1-q^{-1}\mu_\alpha}{1-\mu_\alpha} \cdot \mu_{e_1}^{k_2} \cdot \mu_{e_2}^{k_1} - \frac{1-q^{-1}}{1-\mu_\alpha} \cdot \mu_{e_1}^{k_1} \cdot \mu_{e_2}^{k_2} \right)  \\
& \quad \cdot  q^{-k_2(\frac{R+1}{2}-(R-1) +\overline{\nu})} \cdot  q^{-k_1(\frac{R+1}{2}-R +\overline{\nu})} \cdot q^{-n_\alpha s k_1- n_\alpha s k_2} \\
=& \mca{W}_{\s_{\w_1[0]}}^{\GL_2}(1) \cdot \overline{\mca{W}_{\mca{O}_0}^{\GL_R} }(1) 
\cdot \sum_{\substack{k_2\ge 0 \\ k_1\ge k_2} } \left( \frac{1-q^{-1}\mu_\alpha}{1-\mu_\alpha} \cdot \mu_{e_1}^{k_2} \cdot \mu_{e_2}^{k_1} - \frac{(1-q^{-1})\mu_\alpha}{1-\mu_\alpha} \cdot \mu_{e_1}^{k_1} \cdot \mu_{e_2}^{k_2} \right) \\
& \quad \cdot \mu_{e_2} q^{-(n_\alpha s+ \frac{R+1}{2}-R +\overline{\nu})} \cdot \cdot q^{-k_1(n_\alpha s+ \frac{R+1}{2}-R +\overline{\nu})} \cdot q^{-k_2(n_\alpha s+ \frac{R+1}{2}-(R-1) +\overline{\nu})} .
\end{align*}

Meanwhile, we consider the sum over $Y_{\GL_2, \text{id}}$ and obtain
\begin{align*}
& \sum_{y\in Y_{\GL_2, \text{id}}} \mca{W}_{\s_{\w_1[0]}}^{\GL_2}(\s_y) \cdot \overline{\mca{W}_{\mca{O}_0}^{\GL_R} } (\s_{\phi(y)}) \cdot |\det(\s_y)|^s \cdot \delta_{B_2}^{-1/2}(\s_y) \cdot \delta_{B_R}^{-1/2}(\s_{\phi(y)}) \\
=& \overline{\mca{W}_{\mca{O}_0}^{\GL_R} }(1)\cdot \sum_{ \substack{k_2\ge 0 \\ k_1\ge k_2} } \mca{W}_{\s_{\w_1[0]}}^{\GL_2}(\s_{k_1 n_\alpha e_1 + k_2 n_\alpha e_2}) \cdot  \delta_{B_2}^{-1/2}(\s_{k_1 n_\alpha e_1 + k_2 n_\alpha e_2})  \cdot \overline{\chi}_{e_{R-1}}^{k_2} \overline{\chi}_{e_R}^{k_1} \cdot q^{-n_\alpha s k_1 - n_\alpha s k_2} \\
=& \overline{\mca{W}_{\mca{O}_0}^{\GL_R} }(1) \cdot \mca{W}_{\s_{\w_1[0]}}^{\GL_2}(1) \cdot \sum_{ \substack{k_2\ge 0 \\ k_1\ge k_2} }   \mu_{e_1}^{k_1} \cdot \mu_{e_2}^{k_2} \cdot \overline{\chi}_{e_{R-1}}^{k_2} \overline{\chi}_{e_R}^{k_1} \cdot q^{-n_\alpha s k_1 - n_\alpha s k_2}\\
=& \overline{\mca{W}_{\mca{O}_0}^{\GL_R} }(1) \cdot \mca{W}_{\s_{\w_1[0]}}^{\GL_2}(1) \cdot \sum_{ \substack{k_2\ge 0 \\ k_1\ge k_2} }   \mu_{e_1}^{k_1} \cdot \mu_{e_2}^{k_2}  \cdot q^{-k_1(n_\alpha s+ \frac{R+1}{2}-R +\overline{\nu})} \cdot q^{-k_2(n_\alpha s+ \frac{R+1}{2}-(R-1) +\overline{\nu})} .
\end{align*}
A straightforward computation of combining the above sums over $Y_{\GL_2, \w_1}$ and $Y_{\GL_r, \text{id}}$ gives the desired result as follows.
\begin{prop}  With notations as above, we have
$$\msc{Z}(s, \mca{W}_{\s_{\w_1[0]}}^{\GL_2} \times \overline{\mca{W}_{\mca{O}_0}^{\GL_R} } )=L(n_\alpha s -\frac{R-1}{2}, \mu_{n_\alpha} \times \overline{\nu}) \cdot \mca{W}_{\s_{\w_1[0]}}^{\GL_2}(1) \cdot  \overline{\mca{W}_{\mca{O}_0}^{\GL_R} }(1).$$
\end{prop}

\subsection{Summary for $(\wt{\GL}_2, \wt{\GL}_R)$}
Now we summarise our result in this section.

\begin{thm} \label{T:rank2}
Let $(\wt{\GL}_2, \wt{\GL}_R)$ be a fundamental pair of $n$-fold covering groups satisfying $n|(\p\cdot n_\alpha)$. Let $\pi$ be a generic unramified representation of $\wt{\GL}_2$ such that $I(\mu) \onto \pi$ for some unramified character $\mu$. Let $\Theta(\wt{\GL}_R, \chi)$ be the distinguished theta representation associated to an unramified exceptional $\chi$. Let $\mca{W}^{\GL_R}_{\mca{O}_0}$ be the unique Whittaker model of $\Theta(\wt{\GL}_R, \chi)$. Then for any Whittaker model $\mca{W}^{\GL_r}$ of $\pi$, the following equality holds:
$$\msc{Z} (s, \mca{W}^{\GL_2} \times \overline{\mca{W}_{\mca{O}_0}^{\GL_R}} )= L(n_\alpha s -\frac{R-1}{2}, \mu_{n_\alpha} \times \overline{\nu}) \cdot \mca{W}^{\GL_2}(1) \cdot \overline{ \mca{W}^{\GL_R}_{\mca{O}_0} }(1).$$
Here $\mu_{n_\alpha}$ is the linear character associated to $\mu$ as in (\ref{F:lin}), and $\chi_{n_\alpha}[\nu]$ the linear character to $\chi$ in Lemma \ref{L:exp-ch}.
\end{thm}

\vskip 5pt
\section{Some remarks}
\subsection{Generating function \'a la Bump-Friedberg}
For fundamental pairs of Kazhdan-Patterson coverings, Bump and Friedberg propose in \cite{BF} a general approach of attacking the Bump-Hoffstein conjecture, by considering the generating function of $L(s, \mu_{n_\alpha})$ (i.e. when $\nu_R=\mbf{1}$ in Conjecture \ref{C:GBH}).  We only mention in passing the analogue of the generating function in \cite{BF}, as the formulation is essentially the same.  

Consider a fundamental pair $(\wt{\GL}_r^{(n)}, \wt{\GL}_R^{(n)})$ with $n|(\p\cdot n_\alpha)$. Let $\Theta(\wt{\GL}_R^{(n)}, \chi)$ be the distinguished theta representation associated with $\chi$ such that $\nu_R=\mbf{1}$ in $\chi_{n_\alpha}[\nu_R]$. Let $\mca{W}_{\mca{O}_0}^{\GL_R}$ be the unique Whittaker model of $\Theta(\wt{\GL}_R^{(n)}, \chi)$. Consider the \emph{anti-genuine} function $\wt{\Delta}_s$ on $\wt{\GL}_r^{(n)}$ satisfying:
\begin{enumerate}
\item[$\bullet$] the function $\wt{\Delta}_s: \wt{\GL}_r^{(n)} \to \C$ is $K_r$-biinvariant.
\item[$\bullet$] for $\wt{t}\in \wt{T}_r$, one has $\wt{\Delta}_s(\wt{t})=0$ unless $\wt{t}$ belongs to $\wt{T_{n_\alpha}^\dag}$ and is dominant, in which case
$$\wt{\Delta}_s(\wt{t}) = \val{\det (\wt{t}) }^{\frac{s}{n_\alpha} + \frac{r-1}{n_\alpha}} \cdot \delta_{B_r}(\wt{t})^{\frac{n_\alpha-1}{2n_\alpha}}.$$
\end{enumerate}

Based on this, one has
\begin{conj} [{Bump-Friedberg}]
For any $\wt{g}\in \wt{\GL}_r^{(n)}$, the equality below holds:
$$\int_{U_r} \wt{\Delta}_{n_\alpha s -\frac{R-1}{2}} (u \wt{g}) \cdot \psi(u)\  du =\val{ \det(\wt{g})}^{s-\frac{R-r}{2}} \cdot \overline{ \mca{W}_{\mca{O}_0}^{\GL_R} }(\phi(\wt{g})).$$
\end{conj}
It is shown in \cite{BF} that for Kazhdan-Patterson coverings the above conjecture  implies the Bump-Hoffstein conjecture \ref{C:GBH}. Moreover, in a recent work \cite{Gin}, Ginzburg proves the above statement completely for Kazhdan-Patterson coverings.

We remark that the argument in \cite{BF} readily adapts to general fundamental pairs considered in this paper. Moreover, it is expected that the proof in \cite{Gin} should be applicable to the general case as well.

\subsection{The role of the dual group}
In this last part of the paper, we point out that it is not sufficient to consider a pair $(\wt{\GL}_r^{(n)}, \wt{\GL}_R^{(n)})$ (not necessarily a fundamental pair) such that the dual group of $\wt{\GL}_R^{(n)}$ is  isomorphic $\GL_R$ and $\Theta(\wt{\GL}_R^{(n)}, \chi)$ is distinguished. Namely, we highlight the fact that the two conditions $R=n_\alpha$ and $Y_{\GL_R, Q,n}=n_\alpha \cdot Y_{\GL_R}$ in  ({\sf FP2}) are crucial, essentially because of Lemma \ref{L:crucial}.
\vskip 5pt

The example comes from the case $r=1$ and $R=2$ already. More precisely, we consider the Kazhdan-Patterson three-fold covering group $\wt{\GL}_R:=\wt{\GL}_2$ with  parameter $\p=\bq=-1$, and $\wt{\GL}_r:=\wt{\GL}_1$ the obvious pull-back. By Example \ref{E:KP},
$$Y_{\GL_R, Q, n}= \set{ k_1 e_1 + k_2 e_2: k_1 \equiv k_2 \mod 3  }.$$
Moreover, the dual group of this $\wt{\GL}_2$ is
$$\wt{\GL}_3^\vee=\set{ (g, \lambda) \in \mbf{GL}_2\times \mbf{GL}_1: \det(g)=\lambda^3}.$$
There is an isomorphism
$$\wt{\GL}_3^\vee \simeq \mbf{GL}_2$$
given by $(g, \lambda) \mapsto g\cdot \lambda^{-1}$.

It can be checked easily that $\wp(\OF_{Q,n})= \set{ \wp(\mca{O}_0) }$, and therefore by Proposition \ref{P:dim} the theta representation $\Theta(\wt{\GL}_3, \chi)$ for any exceptional $\chi$ is distinguished. Also, $\wt{\GL}_1^{(3)}$ is just the three-fold covering of $\GL_1$ arising from $Q(e_1)=\p=-1$. In this case, a Whittaker functional for the representation $i(\mu)$ of $\wt{\GL}_1^{(3)}$ is just a linear functional $\lambda: i(\mu) \to \C$, where $\mu\in \Hom_\iota(Z(\wt{\GL}_1), \C^\times)$. Denote by $\mca{W}^{\GL_1}_\lambda$ the the arising Whittaker model, i.e., $\mca{W}^{\GL_1}_\lambda(\wt{g}) = \lambda(i(\mu)(\wt{g}) v_0)$, where $v_0 \in i(\mu)$ is the normalized unramified vector.

We consider
$$ \msc{Z}(s, \mca{W}^{\GL_1}_\lambda \times \overline{\mca{W}_{\mca{O}_0}^{\GL_2} })
:= \int_{\GL_1} \mca{W}_\lambda^{\GL_1}(\wt{g}) \cdot \overline{\mca{W}_{\mca{O}_0}^{\GL_2}}( \phi(\wt{g}))  \cdot  |\det(g)|^{s-\frac{1}{2} } dg $$
which is equal to
$$ \sum_{y\in Y_{\GL_1}}  \mca{W}_\lambda^{\GL_1}(\s_y) \cdot \overline{\mca{W}_{\mca{O}_0}^{\GL_2}}(\s_{\phi(y)})  \cdot  |\det(\s_y)|^s \cdot \delta_{B_2}^{-1/2}(\s_{\phi(y)}).$$

For $y\in Y_{\GL_1}$, it is easy to see that $\mca{W}_{\mca{O}_0}^{\GL_2}(\s_{\phi(y)})=0$ unless $y$ belongs to either  $Y_{\GL_1, \text{id}}=\set{3k \cdot e_1: k\ge 0}$
or
$$Y_{\GL_1, \w_\alpha}= \set{ke_1: k\ge 0 \text{ and } \w_{\alpha}(k e_1) \in \w_\alpha[0] + Y_{\GL_2, Q, n }  }.$$
More explicitly, $Y_{\GL_1, \w_\alpha}= \set{(3k-2)e_1: k\ge 1}$. Therefore we have
\begin{align*}
& \msc{Z}(s, \mca{W}^{\GL_1}_\lambda \times \overline{\mca{W}_{\mca{O}_0}^{\GL_2} })\\
=& \sum_{y\in Y_{\GL_1, \text{id}} \cup Y_{\GL_1, \w_\alpha}}  \mca{W}_\lambda^{\GL_1}(\s_y) \cdot \overline{\mca{W}_{\mca{O}_0}^{\GL_2}}(\s_{\phi(y)})  \cdot  |\det(\s_y)|^s \cdot \delta_{B_2}^{-1/2}(\s_{\phi(y)}).
\end{align*}
Now on the one hand,
\begin{align*}
& \sum_{y\in Y_{\GL_1, \text{id}}}  \mca{W}_\lambda^{\GL_1}(\s_y) \cdot \overline{\mca{W}_{\mca{O}_0}^{\GL_2}}(\s_{\phi(y)})  \cdot  |\det(\s_y)|^s \cdot \delta_{B_2}^{-1/2}(\s_{\phi(y)})  \\
= & \overline{\mca{W}_{\mca{O}_0}^{\GL_2}}(1) \cdot \sum_{k\ge 0} \overline{ \cc_{\mca{O}_0}^{\GL_2} }(\s_{3k e_2}) \cdot \lambda \left( i(\mu)(\s_{3ke_1}) v_0  \right)  \cdot q^{-3ks} \\
= &  \sum_{k\ge 0}  \overline{\chi}_{e_2}^k \cdot (\mu_{e_1})^k \cdot (q^{-3s})^k \cdot \overline{ \cc_{\mca{O}_0}^{\GL_2}(1)  } \cdot \lambda( v_0 ) \\
= & \frac{1}{1- \overline{\chi}_{e_2} \cdot \mu_{e_1} \cdot q^{-3s}} \cdot  \mca{W}_\lambda^{\GL_1}(1) \cdot \overline{\mca{W}_{\mca{O}_0}^{\GL_2}}(1) .
\end{align*}
On the other hand,
\begin{align*}
& \sum_{y\in Y_{\GL_1, \w_\alpha}}  \mca{W}_\lambda^{\GL_1}(\s_y) \cdot \overline{\mca{W}_{\mca{O}_0}^{\GL_2}}(\s_{\phi(y)})  \cdot  |\det(\s_y)|^s \cdot \delta_{B_2}^{-1/2}(\s_{\phi(y)})  \\
= & \overline{\mca{W}_{\mca{O}_0}^{\GL_2}}(1) \cdot \sum_{k\ge 1} \overline{ \cc_{\mca{O}_0}^{\GL_2}}(\s_{\widehat{(3k-2)e_1}}) \cdot \lambda \left( i(\mu)(\s_{(3k-1)e_1}) v_0  \right)  \cdot q^{-(3k-2)s} \\
= & \overline{\mca{W}_{\mca{O}_0}^{\GL_2}}(1) \cdot \sum_{k\ge 1} \overline{ \cc_{\mca{O}_0}^{\GL_2} }(\s_{\w_\alpha[0]}) \cdot
\overline{ \chi }(\s_{-e_1+ (3k-1)e_2}) \cdot
  \lambda \left( i(\mu)(\s_{(3k-1)e_1}) v_0  \right)  \cdot q^{-(3k-2)s}\\
= &   \overline{\mca{W}_{\mca{O}_0}^{\GL_2}}(1) \cdot \overline{ \cc_{\mca{O}_0}^{\GL_2} }(\s_{\w_\alpha[0]}) \cdot \overline{\chi}(\s_{-e_1+ 2e_2}) \cdot q^{-s} \cdot \lambda\left( i(\mu)(\s_{2e_1})v_0 \right) \cdot \sum_{k\ge 0}  \overline{\chi}_{e_2}^k \cdot (\mu_{e_1})^k \cdot (q^{-3s})^k \\
= &  \overline{\mca{W}_{\mca{O}_0}^{\GL_2}}(1) \cdot \frac{q^{-s-1}}{1- \overline{\chi}_{e_2} \cdot \mu_{e_1} \cdot q^{-3s}} \cdot  \g(Q(\alpha^\vee))^{-1} \cdot \overline{ \cc_{\mca{O}_0}^{\GL_r}}(1) \cdot \overline{\chi}(\s_{-e_1+ 2e_2}) \cdot \lambda\left( i(\mu)(\s_{2e_1})v_0 \right) \\
= &   \overline{ \mca{W}_{\mca{O}_0}^{\GL_2}}(1) \cdot \frac{q^{-s-1}}{1- \overline{\chi}_{e_2} \cdot \mu_{e_1} \cdot q^{-3s}} \cdot  \g(Q(\alpha^\vee))^{-1}  \cdot \overline{\chi}(\s_{-e_1+ 2e_2}) \cdot \lambda\left( i(\mu)(\s_{2e_1})v_0 \right) .
\end{align*}

We see that the analogous Bump-Hoffstein conjecture does not hold for this pair $(\wt{\GL}_1^{(3)}, \wt{\GL}_2^{(3)})$ above. In view of this, one would expect that the Bump-Hoffstein conjecture holds for a pair only when it is a fundamental pair in the sense of Definition \ref{D:FP}.



\begin{bibdiv}
\begin{biblist}[\resetbiblist{9999999}]*{labels={alphabetic}}

\bib{BJ}{article}{
   author={Ban, Dubravka},
   author={Jantzen, Chris},
   title={The Langlands quotient theorem for finite central extensions of
   $p$-adic groups},
   journal={Glas. Mat. Ser. III},
   volume={48(68)},
   date={2013},
   number={2},
   pages={313--334},
   issn={0017-095X},
   review={\MR{3151110}},
   doi={10.3336/gm.48.2.07},
}
%

\bib{BD}{article}{
   author={Brylinski, Jean-Luc},
   author={Deligne, Pierre},
   title={Central extensions of reductive groups by $\bold K_2$},
   journal={Publ. Math. Inst. Hautes \'Etudes Sci.},
   number={94},
   date={2001},
   pages={5--85},
   issn={0073-8301},
   review={\MR{1896177}},
   doi={10.1007/s10240-001-8192-2},
}
%
\bib{BF}{article}{
   author={Bump, Daniel},
   author={Friedberg, Solomon},
   title={Metaplectic generating functions and Shimura integrals},
   conference={
      title={Automorphic forms, automorphic representations, and arithmetic},
      address={Fort Worth, TX},
      date={1996},
   },
   book={
      series={Proc. Sympos. Pure Math.},
      volume={66},
      publisher={Amer. Math. Soc., Providence, RI},
   },
   date={1999},
   pages={1--17},
   review={\MR{1703755}},
}
%
\bib{BFG}{article}{
   author={Bump, Daniel},
   author={Furusawa, Masaaki},
   author={Ginzburg, David},
   title={Non-unique models in the Rankin-Selberg method},
   journal={J. Reine Angew. Math.},
   volume={468},
   date={1995},
   pages={77--111},
   issn={0075-4102},
   review={\MR{1361787}},
}
%
\bib{BH1}{article}{
   author={Bump, Daniel},
   author={Hoffstein, Jeffrey},
   title={On Shimura's correspondence},
   journal={Duke Math. J.},
   volume={55},
   date={1987},
   number={3},
   pages={661--691},
   issn={0012-7094},
   review={\MR{904946}},
   doi={10.1215/S0012-7094-87-05533-5},
}
%
\bib{BH2}{article}{
   author={Bump, Daniel},
   author={Hoffstein, Jeffrey},
   title={Some conjectured relationships between theta functions and
   Eisenstein series on the metaplectic group},
   conference={
      title={Number theory},
      address={New York},
      date={1985/1988},
   },
   book={
      series={Lecture Notes in Math.},
      volume={1383},
      publisher={Springer, Berlin},
   },
   date={1989},
   pages={1--11},
   review={\MR{1023915}},
   doi={10.1007/BFb0083566},
}
%
\bib{Bou}{book}{
   author={Bourbaki, Nicolas},
   title={Lie groups and Lie algebras. Chapters 4--6},
   series={Elements of Mathematics (Berlin)},
   note={Translated from the 1968 French original by Andrew Pressley},
   publisher={Springer-Verlag, Berlin},
   date={2002},
   pages={xii+300},
   isbn={3-540-42650-7},
   review={\MR{1890629}},
   doi={10.1007/978-3-540-89394-3},
}
%
\bib{FL}{article}{
   author={Finkelberg, Michael},
   author={Lysenko, Sergey},
   title={Twisted geometric Satake equivalence},
   journal={J. Inst. Math. Jussieu},
   volume={9},
   date={2010},
   number={4},
   pages={719--739},
   issn={1474-7480},
   review={\MR{2684259}},
   doi={10.1017/S1474748010000034},
}
%
\bib{Ga1}{article}{
   author={Gao, Fan},
   title={The Langlands-Shahidi L-functions for Brylinski-Deligne extensions},
   status={preprint, to appear in Amer. J. Math.},
}
%
\bib{Ga2}{article}{
   author={Gao, Fan},
   title={Distinguished theta representations for Brylinski-Deligne covering groups},
   status={preprint, available at https://arxiv.org/abs/1602.01880},
}
%
\bib{GG}{article}{
   author={Gan, Wee Teck},
   author={Gao, Fan},
   title={The Langlands-Weissman program for Brylinski-Deligne extensions},
   status={preprint, available at www.math.nus.edu.sg/~matgwt},
}
%
\bib{Gin}{article}{
   author={Ginzburg, David},
   title={Generating functions on covering groups},
   status={preprint, available at https://arxiv.org/abs/1603.05784v1},
}
%
\bib{GHPS}{article}{
   author={Gelbart, Stephen},
   author={Howe, Roger},
   author={Piatetski-Shapiro, Ilya},
   title={Uniqueness and existence of Whittaker models for the metaplectic
   group},
   journal={Israel J. Math.},
   volume={34},
   date={1979},
   number={1-2},
   pages={21--37 (1980)},
   issn={0021-2172},
   review={\MR{571393}},
   doi={10.1007/BF02761822},
}
%
\bib{KP}{article}{
   author={Kazhdan, D. A.},
   author={Patterson, S. J.},
   title={Metaplectic forms},
   journal={Inst. Hautes \'Etudes Sci. Publ. Math.},
   number={59},
   date={1984},
   pages={35--142},
   issn={0073-8301},
   review={\MR{743816}},
}
%
\bib{Mc1}{article}{
   author={McNamara, Peter J.},
   title={Principal series representations of metaplectic groups over local
   fields},
   conference={
      title={Multiple Dirichlet series, L-functions and automorphic forms},
   },
   book={
      series={Progr. Math.},
      volume={300},
      publisher={Birkh\"auser/Springer, New York},
   },
   date={2012},
   pages={299--327},
   review={\MR{2963537}},
   doi={10.1007/978-0-8176-8334-413},
}
%
\bib{Mc2}{article}{
   author={McNamara, Peter J.},
   title={The metaplectic Casselman-Shalika formula},
   journal={Trans. Amer. Math. Soc.},
   volume={368},
   date={2016},
   number={4},
   pages={2913--2937},
   issn={0002-9947},
   review={\MR{3449262}},
   doi={10.1090/tran/6597},
}
%
\bib{Pat}{article}{
   author={Patterson, S. J.},
   title={Metaplectic forms and Gauss sums. I},
   journal={Compositio Math.},
   volume={62},
   date={1987},
   number={3},
   pages={343--366},
   issn={0010-437X},
   review={\MR{901396}},
}
%
\bib{PSR}{article}{
   author={Piatetski-Shapiro, I.},
   author={Rallis, S.},
   title={A new way to get Euler products},
   journal={J. Reine Angew. Math.},
   volume={392},
   date={1988},
   pages={110--124},
   issn={0075-4102},
   review={\MR{965059}},
   doi={10.1515/crll.1988.392.110},
}
%
\bib{Re}{article}{
   author={Reich, Ryan Cohen},
   title={Twisted geometric Satake equivalence via gerbes on the
   factorizable Grassmannian},
   journal={Represent. Theory},
   volume={16},
   date={2012},
   pages={345--449},
   issn={1088-4165},
   review={\MR{2956088}},
   doi={10.1090/S1088-4165-2012-00420-4},
}
%
\bib{Sha}{book}{
   author={Shahidi, Freydoon},
   title={Eisenstein series and automorphic $L$-functions},
   series={American Mathematical Society Colloquium Publications},
   volume={58},
   publisher={American Mathematical Society, Providence, RI},
   date={2010},
   pages={vi+210},
   isbn={978-0-8218-4989-7},
   review={\MR{2683009}},
   doi={10.1090/coll/058},
}
%
\bib{Sa}{article}{
   author={Savin, Gordan},
   title={A nice central extension of $GL_r$},
   status={preprint},
}
%
\bib{Suz1}{article}{
   author={Suzuki, Toshiaki},
   title={Rankin-Selberg convolutions of generalized theta series},
   journal={J. Reine Angew. Math.},
   volume={414},
   date={1991},
   pages={149--205},
   issn={0075-4102},
   review={\MR{1092629}},
   doi={10.1515/crll.1991.414.149},
}
%
\bib{Suz2}{article}{
   author={Suzuki, Toshiaki},
   title={Metaplectic Eisenstein series and the Bump-Hoffstein conjecture},
   journal={Duke Math. J.},
   volume={90},
   date={1997},
   number={3},
   pages={577--630},
   issn={0012-7094},
   review={\MR{1480547}},
   doi={10.1215/S0012-7094-97-09016-5},
}
%
\bib{Suz3}{article}{
   author={Suzuki, Toshiaki},
   title={Distinguished representations of metaplectic groups},
   journal={Amer. J. Math.},
   volume={120},
   date={1998},
   number={4},
   pages={723--755},
   issn={0002-9327},
   review={\MR{1637947}},
}
%
\bib{Szp}{article}{
   author={Szpruch, Dani},
   title={Uniqueness of Whittaker model for the metaplectic group},
   journal={Pacific J. Math.},
   volume={232},
   date={2007},
   number={2},
   pages={453--469},
   issn={0030-8730},
   review={\MR{2366363}},
   doi={10.2140/pjm.2007.232.453},
}
%
\bib{We1}{article}{
   author={Weissman, Martin H.},
   title={Metaplectic tori over local fields},
   journal={Pacific J. Math.},
   volume={241},
   date={2009},
   number={1},
   pages={169--200},
   issn={0030-8730},
   review={\MR{2485462}},
   doi={10.2140/pjm.2009.241.169},
}
%
\bib{We2}{article}{
   author={Weissman, Martin H.},
   title={L-groups and parameters for covering groups},
   status={preprint, to appear in Ast\'erisque},
}


\end{biblist}
\end{bibdiv}
\end{document}